\documentclass[a4paper,12pt]{amsart}
\usepackage[margin=1in]{geometry}

\usepackage{amsmath}
\usepackage{amssymb}
\usepackage{enumitem}
\usepackage{mathtools}
\usepackage[unicode]{hyperref}
\usepackage{url}
\usepackage{xcolor}
\usepackage{etoolbox}
\usepackage{array}
\usepackage{adjustbox}
\usepackage{braket}
\usepackage{orcidlink}

\usepackage{lineno}

\hypersetup{
  colorlinks   = true, 
  urlcolor     = {blue!50!black}, 
  linkcolor    = {blue!50!black}, 
  citecolor   = {red!50!black} 
}

\usepackage{cleveref}

\theoremstyle{plain}
\newtheorem{thm}{Theorem}[section]

\newtheorem{prop}[thm]{Proposition}
\newtheorem{lem}[thm]{Lemma}
\newtheorem{fact}[thm]{Fact}
\newtheorem*{thm*}{Theorem}
\newtheorem*{prop*}{Proposition}
\newtheorem*{lem*}{Lemma}
\newtheorem*{cor*}{Corollary}
\newtheorem*{fact*}{Fact}

\newtheorem{customthm}{Theorem}
\newenvironment{numthm}[1]
{\renewcommand\thecustomthm{#1}\customthm}
{\endcustomthm}

\theoremstyle{definition}
\newtheorem{defn}[thm]{Definition}
\newtheorem{eg}[thm]{Example}
\newtheorem{notn}[thm]{Notation}

\theoremstyle{remark}
\newtheorem{rem}[thm]{Remark}
\newtheorem*{term*}{Terminology}
\newtheorem{qn}[thm]{Question}
\newtheorem*{qn*}{Question}

\newenvironment{subproof}[1][\proofname]{%
\par\textit{Proof of claim. }}{%
  \hfill $\blacksquare$ \par 
}

\crefname{lem}{lemma}{lemmas}
\Crefname{lemma}{Lemma}{Lemmas}
\crefname{prop}{proposition}{propositions}
\crefname{thm}{theorem}{theorems}
\crefname{defn}{definition}{definitions}
\crefname{rem}{remark}{remarks}
\crefname{fact}{fact}{facts}

\DeclareMathOperator{\Age}{Age}
\DeclareMathOperator{\Aut}{Aut}

\DeclareMathOperator{\ext}{ext}

\DeclareMathOperator{\Th}{Th}
\DeclareMathOperator{\tp}{tp}

\DeclareMathOperator{\Wr}{Wr}

\newcommand{\twopartdef}[4]
{
	\left\{
	\begin{array}{ll}
		#1 & \mbox{if } #2 \\
		#3 & \mbox{if } #4
	\end{array}
	\right.
}
\newcommand{\threepartdef}[6]

\newcommand{\N}{
	\mathbb{N}
}

\newcommand{\mc}[1]{
	\mathcal{#1}
}

\newcommand{\ov}[1]{
	\overline{#1}
}

\newcommand{\ex}{
	\exists\,
}
\newcommand{\sub}{
	\subseteq
}

\newcommand{\fin}{
	\subseteq_{\text{fin.\!}}
}

\newcommand{\Jarik}{Ne\v{s}et\v{r}il }

\newcommand{\NR}{Ne\v{s}et\v{r}il-R\"{o}dl }
\newcommand{\NRn}{Ne\v{s}et\v{r}il-R\"{o}dl}

\newcommand{\Fr}{Fra\"{i}ss\'{e} }
\newcommand{\Frn}{Fra\"{i}ss\'{e}}

\newcommand{\omegaLomega}{\omega^{<\omega}}
\newcommand{\omegaomega}{\omega^{\omega}}

\expandafter\patchcmd\csname\string\proof\endcsname{\normalparindent}{0pt }{}{}

\usepackage{ulem}
\usepackage{todonotes}

\colorlet{NadavColor}{-green!40!yellow}
\colorlet{RobColor}{pink}

\makeatletter
\providecommand\@dotsep{5}
\renewcommand{\listoftodos}[1][\@todonotes@todolistname]{%
  \@starttoc{tdo}{#1}}
\makeatother

\author{Nadav Meir  \orcidlink{0000-0002-7774-2892}}

\address{\parbox{\linewidth}{Nadav Meir\\
Department of Mathematics, Ben-Gurion University of the Negev, \\
84105 Be'er-Sheva, Israel \\
and \\
Instytut Matematyczny, Uniwersytet Wroc{\l}awski\\
pl. Grunwaldzki 2/4 \\ 
50-384 Wroc{\l}aw, Poland
}
}

\email{math@nadav.me}

\thanks{The first author is supported by Israel Science Foundation grant number 555/21, Israel Science Foundation grant number 665/20, and Narodowe Centrum Nauki, Poland, grant 2016/22/E/ST1/00450.
}

\author{Rob Sullivan}

\address{\parbox{\linewidth}{Rob Sullivan\\
Institut f\"{u}r Mathematische Logik und Grundlagenforschung\\
Universit\"{a}t  M\"{u}nster\\  
Einsteinstrasse 62\\
48149  M\"{u}nster, Germany}}

\email{robertsullivan1990+maths@gmail.com}

\thanks{The second author is funded by the Deutsche Forschungsgemeinschaft (DFG, German Research Foundation) under Germany’s Excellence Strategy EXC 2044–390685587, Mathematics Münster: Dynamics–Geometry–Structure and by CRC 1442 Geometry: Deformations and Rigidity.}



\title{The externally definable Ramsey property and fixed points on type spaces}

\date{\today}

\begin{document}

\begin{abstract}
    We discuss the externally definable Ramsey property, a weakening of the Ramsey property for relational structures, where the only colourings considered are those that are externally definable: that is, definable with parameters in an elementary extension. We show a number of basic results analogous to the classical Ramsey theory, and show that, for an ultrahomogeneous structure $M$ with countable age, the externally definable Ramsey property is equivalent to the dynamical statement that, for all $n \in \N$, every subflow of the space $S_n(M)$ of $n$-types has a fixed point. We discuss a range of examples, including results regarding the lexicographic product of structures.
\end{abstract}
      
\maketitle

\section{Introduction}

    The paper \cite{KPT05} of Kechris, Pestov and Todor\v{c}evi\'{c} established links between structural Ramsey theory and the topological dynamics of automorphism groups of countable ultrahomogeneous structures (the \emph{KPT correspondence}). Recently, the results of \cite{KPT05} (and generalisations - see, for example, \cite{NVT13} and \cite{Zuc16}) have been considered from a more model-theoretic perspective by various authors. In \cite{KP23}, key results from \cite{KPT05}, \cite{NVT13} and \cite{Zuc16} are reproved using model-theoretic techniques, and generalised to the case of possibly uncountable ultrahomogeneous structures. The paper \cite{Hru19} also considers Ramsey properties from the perspective of first-order logic, and \cite{HKP20} introduces the notion of (extremely) amenable first-order theories.

    In \cite{KLM}, particularly relevant in terms of the background to this paper, the authors define a strongly model-theoretic notion of \emph{(externally) definable colourings}. Various \emph{definable Ramsey properties} are defined by restricting attention to only those colourings that are first-order (externally) definable, and where structures embed via elementary embeddings. The paper \cite{KLM} thoroughly studies these definable Ramsey properties for first-order theories, phrased in terms of Ramsey properties for definable colourings of the monster model (or equivalently an $\aleph_0$-saturated, strongly $\aleph_0$-homogeneous model), and proves several KPT-correspondence-type results. These give dynamical characterisations of definable Ramsey properties of a theory, phrased in terms of the Ellis semigroup of the space of types over a monster model (see Theorem 1 of \cite{KLM}), and these characterisations have important consequences: certain definable Ramsey-type properties of theories imply that their Ellis groups are profinite (or even trivial). 
    
    The current paper will take a perspective more inspired by classical structural Ramsey theory: we consider Ramsey properties of structures, rather than theories, and we allow all embeddings. We are concerned with ultrahomogeneous structures, and do not assume $\aleph_0$-saturation. Our central aim is to investigate the \emph{externally definable Ramsey property}, where we restrict attention to colourings of a structure $M$ that are \emph{externally definable}: informally, where there exists a tuple in some $N \succeq M$ such that each colour set is defined by a formula with parameters from the tuple (see \Cref{extdefcoldef}). (This will resemble the definitions of externally definable colouring and externally definable elementary embedding Ramsey property (EDEERP) found in \cite[Definition 3.1 and Definition 3.11]{KLM}.) We work entirely in the context of countable relational languages. 
    
    We consider the following questions:
    \begin{enumerate}
        \item Which ultrahomogeneous structures have the externally definable Ramsey property?
        \item What is the dynamical content of the externally definable Ramsey property for ultrahomogeneous structures in general?
    \end{enumerate}
    We first show (in \Cref{basicpropertiessec}) that the externally definable Ramsey property satisfies similar basic properties to those of the classical Ramsey property, with the notable exception that the equivalence between the Ramsey property for an ultrahomogeneous structure and that for its age is only shown to hold for colourings externally defined by a \emph{finite} set of formulae (\Cref{Ramseyforage}). (The authors believe this latter result and its use in the proofs in this paper form a new approach.) We use these properties to prove that, under the assumption of QE, if we wish to show the externally definable Ramsey property for an ultrahomogeneous structure, it suffices to check externally definable colourings of substructures of size less than the maximum arity of the language (see \Cref{EDRPsimplecheck} - this resembles \cite[Proposition 3.17]{KLM}, though our proof techniques differ, and as we consider structures rather than theories, we do not assume that the structure is $\aleph_0$-saturated). In \Cref{exsec} we then apply \Cref{EDRPsimplecheck} to a range of examples, in an attempt to answer the first question above. A particular class of new unsaturated examples is given by lexicographic products of structures, in \Cref{lexsec}.

    In \Cref{mainthmsec} we answer the second question. It is well-known that each Stone space $S_n(M)$ of $n$-types with parameters from $M$ gives an $\Aut(M)$-flow by acting via translation of parameters (\Cref{typeflow}). We say that a structure $M$ has the \emph{fixed points on type spaces property} (FPT) if, for all $n \geq 1$, every subflow of $S_n(M)$ has a fixed point. We then have the following theorem:

    \begin{numthm}{\ref{FPTiffEDRP}}
        Let $M$ be an ultrahomogeneous structure.
			
        Then:
        \begin{enumerate}
            \item if $M$ has FPT, then $M$ has the externally definable Ramsey property;
            \item if $\Age(M)$ is countable, then $M$ has the externally definable Ramsey property iff $M$ has FPT.
        \end{enumerate}
    \end{numthm}

    This provides an analogue to the fundamental result of the KPT correspondence, \cite[Theorem 4.5]{KPT05}. In the case that $M$ is $\aleph_0$-saturated, \Cref{FPTiffEDRP} follows from \cite[Theorem 3.14]{KLM} (personal communication, K.\ Krupiński). The novelty here is that \Cref{FPTiffEDRP} applies to all ultrahomogeneous $M$ with countable age -- in particular, all \Fr structures, including those that are unsaturated. The dynamical characterisation of the externally definable Ramsey property as FPT is also new.

    The classification programme initiated by \Jarik (\cite{Nes05}) of ultrahomogeneous structures with the Ramsey property may also be interesting in the context of external definability, as shown by the examples in \Cref{exsec}. We refer the reader to \Cref{furtherqnssec} for suggestions of further research directions.

    The authors would also like to note that in \cite{NVT17}, certain definable Ramsey properties for \Fr structures are also considered from a structural Ramsey perspective, and one particular definable Ramsey property is shown to be equivalent to having fixed points on every subflow of the Roelcke compactification of the automorphism group of the structure (see \cite[Definition 1.3 and Theorem 1.8]{NVT17}). The relationship of these results to the results of the present paper is currently unclear (it would be interesting to know if there is any overlap in the $\aleph_0$-categorical case).

    \textbf{Acknowledgements.} The authors would like to thank Krzysztof Krupiński for explicitly pointing out the many connections between this paper and \cite{KLM}. This project began from the perspective of investigating FPT, but the authors believe that the idea of considering definable Ramsey properties (in particular the definition of an externally definable colouring) as found in \cite{KLM} was undoubtedly a significant and helpful influence when proving the main result of this paper, \Cref{FPTiffEDRP}. The second author would also like to thank David Evans, his PhD supervisor, for numerous productive discussions regarding FPT during his PhD (see \cite{Sul23} for the paper that resulted from these discussions, and Chapter 3 of the PhD thesis \cite{Sul22} of the second author).

\section{Background and notation}

    In this section, we review background material on \Fr theory, Ramsey theory and the KPT correspondence. The treatment here is very brief: for further details, see the references throughout.
    
    In this paper, all first-order languages considered are assumed to be countable and relational.

\subsection{Ultrahomogeneity; \Fr theory}

    We assume the reader is familiar with the classical \Fr theory, a clear exposition of which may be found in \cite[Chapter 7]{Hod93}, and here we simply recall \Frn's theorem while establishing some terminology and notation.
    
    \begin{defn}
        Let $\mc{K}$ be a class of $\mc{L}$-structures. We say that $\mc{K}$ is \emph{countable (up to isomorphism)} if the structures in $\mc{K}$ lie in countably many isomorphism classes. We will usually mildly abuse terminology and just say that $\mc{K}$ is \emph{countable}.
    \end{defn}

    \begin{defn}
        We call a countably infinite ultrahomogeneous structure $M$ a \emph{\Fr structure}.
        
        Let $\mc{K}$ be a class of finite $\mc{L}$-structures with the hereditary property (HP), joint embedding property (JEP) and amalgamation property (AP), and assume additionally that:
        \begin{itemize}
            \item $\mc{K}$ is countable;
            \item $\mc{K}$ contains structures of arbitrarily large finite size.
        \end{itemize}
        Then we call $\mc{K}$ a \emph{\Fr class} or \emph{amalgamation class}.
    \end{defn}

    \begin{thm}[\Frn, \cite{Fra54}]
        We have that:
        \begin{itemize}
            \item if $M$ is an ultrahomogeneous structure, then $M$ has the extension property (EP);
            \item if $M$ is countable and has the extension property, then $M$ is ultrahomogeneous;
            \item if $M$ is a \Fr structure, then $\Age(M)$ is a \Fr class;
            \item if $\mc{K}$ is a \Fr class, then there exists a \Fr structure $M$ with $\Age(M) = \mc{K}$, and such an $M$ is unique up to isomorphism.
        \end{itemize}
    \end{thm}

\subsection{Quantifier elimination and \texorpdfstring{$\aleph_0$}{0א}-categoricity}

    We briefly review some background on quantifier elimination and $\aleph_0$-categoricity in the context of ultrahomogeneous structures. We assume the reader is familiar with the Ryll-Nardzewski theorem (see, for example, \cite[Theorem 4.4.1]{Mar02}).
  
    \begin{lem} \label{finLomega}
        Let $\mc{L}$ be a finite relational language, and let $M$ be a \Fr $\mc{L}$-structure. Then $\Th(M)$ is $\aleph_0$-categorical. 
    \end{lem}
    This result follows immediately from the Ryll-Nardzewski theorem, but can also be proved directly (see \cite[Proposition 1.10]{Eva13}).
    
    \begin{lem} \label{omegahomogiffQE}
        Let $\mc{L}$ be a countable relational language, and let $M$ be a countably infinite $\mc{L}$-structure with $\aleph_0$-categorical theory $\Th(M)$. Then $M$ is ultrahomogeneous iff $\Th(M)$ has QE.
    \end{lem}
    This is a corollary of the Ryll-Nardzewski theorem. A proof can be found in \cite[Theorem 1.15]{Eva13}. By combining \Cref{finLomega,omegahomogiffQE}, we obtain the following:
    
    \begin{lem} \label{finLQE}
        Let $\mc{L}$ be a finite relational language, and let $M$ be a \Fr $\mc{L}$-structure. Then $\Th(M)$ has QE.
    \end{lem}

\subsection{Ramsey theory; indivisibility}
  
    \begin{notn}
        Let $A, B$ be $\mc{L}$-structures, where $A$ is finite, and let $\bar{a} = (a_0, \cdots, a_{l-1})$ be an ordering of the elements of $A$. If $A \sub B$, we will slightly abuse notation and often write $\bar{a} \sub B$.
			
        Let $\mc{K}$ be a class of $\mc{L}$-structures. We will often write $\bar{a} \in \mc{K}$ to mean a structure $A \in \mc{K}$ with an associated ordering $\bar{a}$ of its elements.
    \end{notn}
		
    \begin{notn}
        Let $A, B$ be $\mc{L}$-structures, where $A$ is finite, and let $\bar{a} = (a_0, \cdots, a_{l-1})$ be an ordering of the elements of $A$. We write $\binom{B}{\bar{a}}$ to denote the set \[\binom{B}{\bar{a}} = \{f(\bar{a}) \sub B : f : A \to B \text{ an embedding}\}. \]
			
        The particular ordering on $A$ is unimportant - it is merely a convenient way to distinguish embeddings $A \to B$.
    \end{notn}
		
    \begin{notn}
        Let $A, B$ be $\mc{L}$-structures. We write $\binom{B}{A}$ to denote the set \[\binom{B}{A} = \{A' \sub B : A' \cong A\},\] the set of isomorphic copies of $A$ in $B$. 
    \end{notn}

    \begin{defn}
        Let $k \in \N_+$. Let $A, B$ be $\mc{L}$ structures. A \emph{$k$-colouring} of $\binom{B}{A}$ is a function $\chi : \binom{B}{A} \to k$.

        Let $S \sub B$. We say that $S$ is \emph{monochromatic} in $\chi$ if $\chi$ is constant on the set $\binom{S}{A}$.

        Let $\bar{a}$ be an ordering of the elements of $A$. A \emph{$k$-colouring} of $\binom{B}{\bar{a}}$ is a function $\chi : \binom{B}{\bar{a}} \to k$. The definition of monochromatic subsets is analogous.
    \end{defn}

    \begin{notn}
        Let $k \in \N_+$. Let $A, B, C$ be $\mc{L}$-structures, and let $\bar{a}$ be an ordering of the elements of $A$. If for any $k$-colouring $\chi : \binom{C}{\bar{a}} \to k$ there exists a $\chi$-monochromatic copy $B' \in \binom{C}{B}$ of $B$, then we write $C \to (B)^{\bar{a}}_k$. This is called \emph{Erd\H{o}s-Rado arrow notation}.
    \end{notn}

    \begin{defn} \label{RP def}
        Let $\mc{K}$ be a class of finite $\mc{L}$-structures which is closed under isomorphisms.

        Let $\bar{a} \in \mc{K}$. We say that $\mc{K}$ has the \emph{$\bar{a}$-Ramsey property ($\bar{a}$-RP)} if for all $k \in \N_+$, for all $B \in \mc{K}$, there exists $C \in \mc{K}$ with $C \to (B)^{\bar{a}}_k$. 
        
        We say $\mc{K}$ has the \emph{point-Ramsey property} or \emph{vertex-Ramsey property} if $\mc{K}$ has the $\bar{a}$-RP for all $\bar{a} \in \mc{K}$ with $|\bar{a}| = 1$. (Usually we will only use this terminology when all structures in $\mc{K}$ of size $1$ are isomorphic.) 
        
        If $\mc{K}$ has the $\bar{a}$-Ramsey property for all $\bar{a} \in \mc{K}$, then we say that $\mc{K}$ has the \emph{Ramsey property}. 
        (This is also referred to as the \emph{embedding Ramsey property}.)
    \end{defn}

    We now discuss a number of basic lemmas regarding the $\bar{a}$-Ramsey property. We will later prove analogues of these lemmas for the externally definable $\bar{a}$-Ramsey property.

    It suffices to check the $\bar{a}$-Ramsey property for colourings with only two colours (the proof is via a simple induction - see \cite[Lemma 2.5]{Bod15}):
    
    \begin{lem}
        Let $\mc{K}$ be a class of finite $\mc{L}$-structures which is closed under isomorphisms. Let $\bar{a} \in \mc{K}$.

        If for all $B \in \mc{K}$ there exists $C \in \mc{K}$ with $C \to (B)^{\bar{a}}_2$, then $\mc{K}$ has the $\bar{a}$-RP.
    \end{lem}
    
    An analogue of the above lemma will hold for the externally definable Ramsey property: this is \Cref{EDRPkiff2}.
    
    The following result (from \cite{Nes89}) indicates the connection between the (full, unrestricted) Ramsey property and ultrahomogeneous structures:
    
    \begin{lem} \label{RPimpliesAP}
        Let $\mc{K}$ be a class of finite structures which is closed under isomorphisms and has JEP and the Ramsey property. Then $\mc{K}$ has the amalgamation property.
    \end{lem}
    
    It is therefore usual to study the Ramsey property in the context of \Fr structures.
    
    \begin{defn}
        Let $M$ be a structure. Let $\bar{a} \in \Age(M)$. If for all $k \in \N_+$, for all $B \in \Age(M)$, we have that $M \to (B)^{\bar{a}}_k$, then we say that $M$ has the \emph{$\bar{a}$-Ramsey property}.
        
        As in \Cref{RP def}, we say that $M$ has the \emph{point-Ramsey property} if $M$ has the $\bar{a}$-Ramsey property for all $\bar{a} \in \Age(M)$ with $|\bar{a}| = 1$, and we say that $M$ has the \emph{Ramsey property} if $M$ has the $\bar{a}$-Ramsey property for all $\bar{a} \in \Age(M)$. 
    \end{defn}
    
    The following lemma shows that to check the Ramsey property on a \Fr structure, it is equivalent to check the Ramsey property on its age. This is usually referred to as the ``compactness principle" (see \cite[Section 1.5]{GRS90} or \cite[Proposition 2.8]{Bod15}).
    
    \begin{lem} \label{classicalRamseyiffage}
        Let $M$ be a \Fr structure. Let $k \in \N_+$. Let $\bar{a} \in \Age(M)$, $B \in \Age(M)$. Then $M \to (B)^{\bar{a}}_k$ iff there exists $C \in \Age(M)$ with $C \to (B)^{\bar{a}}_k$. 
    \end{lem}
    
    The analogue of the above lemma for the externally definable Ramsey property, \Cref{Ramseyforage}, will present significant differences.
    
    We now consider a number of indivisibility properties (see \cite{ES91}), which will be useful in proving the point-Ramsey property for particular examples.
    
    \begin{defn} \label{pigeonindivdef}
        Let $M$ be a structure.
        \begin{itemize}
            \item If for every partition $M = M_1 \cup M_2$ of the vertices of $M$, we have $M_1 \cong M$ or $M_2 \cong M$, we say that $M$ has the \emph{pigeonhole property} (following \cite{Cam97}).
            \item If for every partition $M = M_1 \cup M_2$ of the vertices of $M$, we have that $M_1$ contains an isomorphic copy of $M$ or $M_2$ contains an isomorphic copy of $M$, we say that $M$ is \emph{indivisible}.
            \item If for every partition $M = M_1 \cup M_2$ of the vertices of $M$, we have $\Age(M) \sub \Age(M_1)$ or $\Age(M) \sub \Age(M_2)$, we say that $M$ is \emph{age-indivisible}. 
        \end{itemize}
    \end{defn}
    
    The proof of the following lemma is immediate.
    
    \begin{lem} \label{pigeonindivlem}
        Let $M$ be a structure.
        \begin{enumerate}
            \item If $M$ has the pigeonhole property, then $M$ is indivisible.
            \item If $M$ is indivisible, then $M$ is age-indivisible.
        \end{enumerate}
    \end{lem}
    The following lemma is folklore.
    
    \begin{lem} \label{ageindivpointRamsey}
        Let $M$ be a countable transitive structure (i.e.\ $\Aut(M) \curvearrowright M$ is transitive). Then $M$ is age-indivisible iff $M$ has the point-Ramsey property.
    \end{lem}
    \begin{proof}
        As $M$ is transitive, given $\bar{a} \in \Age(M)$ with $|\bar{a}| = 1$, a colouring of the points of $M$ is the same as a colouring of $\binom{M}{\bar{a}}$, and vice versa.
    
        $\Rightarrow:$ Immediate. $\Leftarrow:$ Let $(S, T)$ be a partition of $M$. Let $\chi : M \to 2$ be the indicator function of $T$. Write $M = \bigcup_{i < \omega} B_i$ as the union of an increasing chain $B_0 \sub B_1 \sub \cdots$ of finite substructures of $M$. Then, as $M$ has the point-Ramsey property, for each $i < \omega$, $M$ contains a $\chi$-monochromatic copy of $B_i$. By the pigeonhole principle, there exists a colour $r \in \{0, 1\}$ such that there are infinitely many $i$ with a monochromatic copy of $B_i$ in $M$ of colour $r$, and so $\chi^{-1}(r)$ contains a copy of each element of $\Age(M)$.
    \end{proof}

\subsection{Flows and the KPT correspondence}

    A central object in topological dynamics is the following (see \cite{Aus88} for further background):

    \begin{defn}
        Let $G$ be a Hausdorff topological group. A \emph{$G$-flow} is a continuous action $G \curvearrowright X$ of $G$ on a non-empty compact Hausdorff topological space $X$.
    \end{defn}

    \begin{defn}
        Let $G \curvearrowright X$ be a $G$-flow. A \emph{subflow} of this $G$-flow consists of the restriction of the continuous action $G \curvearrowright X$ to a non-empty $G$-invariant compact subset $Y \sub X$.
    \end{defn}

    The paper \cite{KPT05} established a link between the following dynamical property and Ramsey theory:

    \begin{defn}
        A Hausdorff topological group $G$ is said to be \emph{extremely amenable} if every $G$-flow has a $G$-fixed point.
    \end{defn}

    \begin{thm}[{\cite[Theorem 4.5]{KPT05}}]
        Let $M$ be a \Fr structure. Then $\Aut(M)$ is extremely amenable iff $M$ has the Ramsey property.
    \end{thm}

    The above theorem was a key source of motivation for the main theorem of this paper, \Cref{FPTiffEDRP}.

\section{The externally definable Ramsey property}

\subsection{Externally definable colourings} \label{basicpropertiessec}
    \begin{defn}
        Let $M$ be a structure. Let $\bar{a} \in \Age(M)$. Let $\Gamma$ be a set of $k$-colourings of $\binom{M}{\bar{a}}$. 

        Let $B \in \Age(M)$. If we have that, for every colouring $\chi \in \Gamma$, there exists $B' \in \binom{M}{B}$ with $\binom{B'}{\bar{a}}$ monochromatic in $\chi$, then we write \[M \to_\Gamma (B)^{\bar{a}}_k,\] as an adaptation of the usual Erd\H{o}s-Rado arrow notation.
   
        We say that $M$ has the \emph{$\bar{a}$-Ramsey property for $\Gamma$ ($\bar{a}$-RP for $\Gamma$)} if we have that for all $B \in \Age(M)$, $M \to_\Gamma (B)^{\bar{a}}_k$.

        Let $C \in \Age(M)$. If for any $C' \in \binom{M}{C}$ and for any $\chi \in \Gamma$, there is $B' \in \binom{C'}{B}$ with $\binom{B'}{\bar{a}}$ monochromatic in $\chi$, then we write \[C \to_\Gamma (B)^{\bar{a}}_k,\] again adapting the usual Erd\H{o}s-Rado arrow notation. (Note that here $\Gamma$ is a set of colourings of $M$, and we consider the colourings induced by $\Gamma$ on each copy of $C$ in $M$. See \Cref{witnessCthereexists}.)
   
        We say that $\Age(M)$ has the \emph{$\bar{a}$-Ramsey property for $\Gamma$ ($\bar{a}$-RP for $\Gamma$)} if, for each $B \in \Age(M)$, there exists $C \in \Age(M)$ such that $C \to_\Gamma (B)^{\bar{a}}_k$.
    \end{defn}

    \begin{defn} \label{extdefcoldef}
        Let $M$ be a structure. Let $\bar{a} \in \Age(M)$, and let $k \in \N_+$. A colouring $\chi : \binom{M}{\bar{a}} \to k$ is \emph{externally definable} if there exist:
			
        \begin{itemize} 
            \item an elementary extension $N \succeq M$;
            \item a tuple $\bar{e} \in N$;
            \item $k$ formulae $\phi_0(\bar{x}, \bar{y}), \cdots, \phi_{k-1}(\bar{x}, \bar{y})$;
        \end{itemize}	
				
        such that, for all $\bar{a}' \in \binom{M}{\bar{a}}$ and $i < k$, \[\chi(\bar{a}') = i \Leftrightarrow N \models \phi_i(\bar{e}, \bar{a}').\]
			
        In the simpler case $k = 2$, we may externally define a colouring $\chi : \binom{M}{\bar{a}} \to 2$ by a single formula $\phi(\bar{x}, \bar{y})$ corresponding to the elements of $\binom{M}{\bar{a}}$ of colour $0$.
    \end{defn}

    We note that the definition of an externally definable colouring resembles that of \cite[Definition 3.1(b)]{KLM} (see in particular \cite[Remark 3.3]{KLM}), except that we allow all embeddings in $\binom{M}{\bar{a}}$ rather than just elementary ones.

    \begin{defn}
        Let $M$ be a structure. Let $\bar{a} \in \Age(M)$, and let $k \in \N_+$. Let $\Phi$ be a set of formulae with free variables among $\bar{x}, \bar{y}$ (where $|\bar{y}| = |\bar{a}|$). 
        
        We write $\ext(\Phi, \binom{M}{\bar{a}}, k)$ for the set of $k$-colourings of $\binom{M}{\bar{a}}$ externally definable in some elementary extension of $M$ by $k$ formulae from $\Phi$.

        We write $\ext(\binom{M}{\bar{a}}, k)$ for the set of $k$-colourings of $\binom{M}{\bar{a}}$ which are externally definable.

        When using Erd\H{o}s-Rado arrow notation, we will abbreviate and write $\ext(\Phi)$ and $\ext$, as $M$, $\bar{a}$ and $k$ are clear from context (e.g.\ we write $M \to_{\ext} (B)^{\bar{a}}_k$).

        If $M \to_{\ext} (B)^{\bar{a}}_k$ for all $B \in \Age(M)$, we say that $M$ has the \emph{externally definable $\bar{a}$-Ramsey property (externally definable $\bar{a}$-RP)}. If $M$ has the externally definable $\bar{a}$-Ramsey property for all $\bar{a} \in \Age(M)$, we say that $M$ has the \emph{externally definable Ramsey property}.
    \end{defn}
  
    \begin{lem} \label{shiftedcbyaut}
        Let $M$ be a structure. Let $\bar{a} \in \Age(M)$. Let $\chi : \binom{M}{\bar{a}} \to k$ be a colouring which is externally definable by the formulae $\phi_0(\bar{x}, \bar{y}), \cdots, \phi_{k-1}(\bar{x}, \bar{y})$. Let $g \in \Aut(M)$. Then the colouring $g \cdot \chi : \binom{M}{\bar{a}} \to k$ defined by $(g \cdot \chi)(\bar{a}') = \chi(g^{-1}\bar{a}')$ is also externally definable by the same formulae $\phi_0(\bar{x}, \bar{y}), \cdots, \phi_{k-1}(\bar{x}, \bar{y})$.
    \end{lem}
    \begin{proof}
        Let $\bar{e} \in N \succeq M$ be an $|\bar{x}|$-tuple in an elementary extension $N$ of $M$ externally defining $\chi$ by the formulae $\phi_0(\bar{x}, \bar{y}), \cdots, \phi_{k-1}(\bar{x}, \bar{y})$. Let $p(\bar{x}) = \tp_N(\bar{e}/M)$. Then $g \cdot p(\bar{x})$ is the type of a tuple externally defining $g \cdot \chi$ by the formulae $\phi_0(\bar{x}, \bar{y}), \cdots, \phi_{k-1}(\bar{x}, \bar{y})$.
    \end{proof}

    \begin{lem} \label{witnessCthereexists}
        Let $M$ be an ultrahomogeneous structure. Let $k \in \N_+$. Let $\bar{a} \in \Age(M)$ and $B \in \Age(M)$. Let $\Phi$ be a set of formulae with free variables among $\bar{x}, \bar{y}$ (where $|\bar{y}| = |\bar{a}|$). Suppose that $C \in \Age(M)$ is such that there exists $C_0 \in \binom{M}{C}$ with the property that, for all $\chi \in \ext(\Phi, \binom{M}{\bar{a}}, k)$, there is a $\chi$-monochromatic copy of $B$ inside $C_0$. Then $C \to_{\ext(\Phi)} (B)^{\bar{a}}_k$.
    \end{lem}

    \begin{proof}
        Let $C' \in \binom{M}{C}$, and let $\chi \in \ext(\Phi, \binom{M}{\bar{a}}, k)$. 
        Then, as $M$ is ultrahomogeneous, there exists $g \in \Aut(M)$ with $gC' = C_0$. By \Cref{shiftedcbyaut}, we have that $g \cdot \chi \in \ext(\Phi, \binom{M}{\bar{a}}, k)$, so $C_0$ contains a $(g \cdot \chi)$-monochromatic copy $B_0$ of $B$. 
        Therefore $g^{-1}B_0$ is a $\chi$-monochromatic copy of $B$ in $C'$. 
    \end{proof}

    We now establish a key correspondence between externally definable Ramsey properties of an ultrahomogeneous structure and externally definable Ramsey properties of its age, which we shall use frequently in the remainder of this article (this approach is new, to the best of our knowledge). Note that the below proposition is a more restricted result than the analogue in the standard Ramsey theory (\Cref{classicalRamseyiffage}), and specifically only applies to \emph{finite} sets of formulae.
  
    \begin{prop} \label{Ramseyforage}
        Let $M$ be an ultrahomogeneous structure. Let $\bar{a} \in \Age(M)$ and let $B \in \Age(M)$. Let $k \in \N_+$. Let $\Phi$ be a finite set of formulae with free variables among $\bar{x}, \bar{y}$ (where $|\bar{y}| = |\bar{a}|$).
			
        Then $M \to_{\ext(\Phi)} (B)^{\bar{a}}_k$ iff there exists $C \in \Age(M)$ with $C \to_{\ext(\Phi)} (B)^{\bar{a}}_k$.
    \end{prop}
  
    \begin{proof}
        $\Leftarrow:$ Immediate. $\Rightarrow:$ We prove the contrapositive. Recall that the space $k^{\binom{M}{\bar{a}}}$ is compact. Let $(\phi_j(\bar{x}, \bar{y}))_{j < k}$ be $k$ formulae from $\Phi$, and let $V$ consist of the colourings $\binom{M}{\bar{a}} \to k$ which are externally definable by $(\phi_j(\bar{x}, \bar{y}))_{j < k}$.

        \textbf{Claim:} $V$ is closed in $k^{\binom{M}{\bar{a}}}$.
        \begin{subproof}
            Take $\chi_0 \in k^{\binom{M}{\bar{a}}} \setminus V$. As $\chi_0$ is not externally definable by $(\phi_j(\bar{x}, \bar{y}))_{j < k}$, we have that for any $N \succeq M$ and any $|\bar{x}|$-tuple $\bar{e} \in N$, there exist $i < k$ and $\bar{a}' \in \binom{M}{\bar{a}}$ such that:
            \begin{itemize}
                \item $\chi_0(\bar{a}') = i$ and $N \not\models \phi_i(\bar{e}, \bar{a}')$; or
                \item $\chi_0(\bar{a}') \neq i$ and $N \not\models \neg\phi_i(\bar{e}, \bar{a}')$.
            \end{itemize}
            
            We define a set of $\mc{L}(M)$-formulae $\Psi(\bar{x})$ as follows. For $i < k$ and $\bar{a}' \in \binom{M}{\bar{a}}$:
            \begin{itemize}
                \item $\phi_i(\bar{x}, \bar{a}') \in \Psi(\bar{x})$ if there exists $N \succeq M$ and $\bar{e} \in N$ such that $\chi_0(\bar{a}') = i$ and $N \not\models \phi_i(\bar{e}, \bar{a}')$;
                \item $\neg\phi_i(\bar{x}, \bar{a}') \in \Psi(\bar{x})$ if there exists $N \succeq M$ and $\bar{e} \in N$ such that $\chi_0(\bar{a}') \neq i$ and $N \not\models \neg\phi_i(\bar{e}, \bar{a}')$.
            \end{itemize}
            
            Thus we have that $\Psi(\bar{x}) \,\cup\, \Th_{\mc{L}(M)}(M)$ is inconsistent, and so by compactness, there exists a finite subset $\Psi_0(\bar{x}) = \{\psi^{}_r(\bar{x}, \bar{a}'_r)\}^{}_{r < s} \sub \Psi(\bar{x})$ such that $\Psi_0(\bar{x}) \,\cup\, \Th_{\mc{L}(M)}(M)$ is inconsistent. Let $U \sub k^{\binom{M}{\bar{a}}}$ be the basic open set given by \[U = \{\chi \in k^{\binom{M}{\bar{a}}} : \chi(\bar{a}'_r) = \chi_0(\bar{a}'_r) \text{ for each } r < s\}.\] Take $\chi \in U$. If it were the case that $\chi$ was externally definable by $(\phi_j(\bar{x}, \bar{y}))_{j < k}$, then this would give a model of $\Psi_0(\bar{x}) \,\cup\, \Th_{\mc{L}(M)}(M)$. So $U \sub k^{\binom{M}{\bar{a}}} \setminus V$, and thus $V$ is closed.
        \end{subproof}
        Let $V_\Phi := \ext(\Phi, \binom{M}{\bar{a}}, k)$. Then as $\Phi$ is finite, we have that $V_\Phi$ is a finite union of closed sets, and hence closed. For $C \fin M$, let $W_C$ denote the set of colourings $\chi : \binom{M}{\bar{a}} \to k$ such that $C$ contains no $\chi$-monochromatic copy of $B$. It is easy to check that each $W_C$ is closed, and that for $C_0, C_1 \fin M$, we have $W_{C_0 \cup C_1} \sub W_{C_0} \cap W_{C_1}$. For $C \fin M$, let $V_C = V_\Phi \cap W_C$. By assumption (using \Cref{witnessCthereexists}), each $V_C$ is non-empty. Therefore the collection of closed sets $\{V_C \mid C \fin M\}$ has the finite intersection property, and as $k^{\binom{M}{\bar{a}}}$ is compact, we have that $\bigcap_{C \fin M} V_C$ is non-empty. Thus we have $M \not\to_{\ext(\Phi)} (B)^{\bar{a}}_k$.
    \end{proof}

    \begin{lem} \label{EDRPkiff2}
        Let $M$ be an ultrahomogeneous structure. Let $\bar{a} \in \Age(M)$ and $B \in \Age(M)$. Let $k \geq 2$. Then $M \to_{\ext} (B)^{\bar{a}}_k$ iff $M \to_{\ext} (B)^{\bar{a}}_2$.
    \end{lem}
    \begin{proof}
        $\Rightarrow:$ this is immediate: given an externally definable $2$-colouring, we may externally define it as a $k$-colouring by externally defining the remaining sets of colours using $\bot$.
        
        $\Leftarrow:$ We use induction on $k \geq 2$, and the base case is given by assumption. Suppose, for $l \leq k$, that $M \to_{\ext} (B)^{\bar{a}}_l$. Let $\chi : \binom{M}{\bar{a}} \to k+1$ be externally defined by $\bar{e} \in N \succeq M$ and $\phi_0(\bar{x}, \bar{y}), \cdots, \phi_k(\bar{x}, \bar{y})$. We want to show that $M$ has a $\chi$-monochromatic copy of $B$.
			
        Let $\chi_1 : \binom{M}{\bar{a}} \to k$ be the colouring:
			
        \begin{equation*}
            \chi_1(\bar{a}') =
            \begin{cases}
                \chi(\bar{a}') & \chi(\bar{a}') = 0, \cdots, k-2 \\
                k-1 & \chi(\bar{a}') = k-1, k
            \end{cases}
        \end{equation*}
			
        Then $\chi_1$ is externally defined by $\bar{e} \in N \succeq M$ and the formulae \[\phi_0(\bar{x}, \bar{y}),\, \cdots,\, \phi_{k-2}(\bar{x}, \bar{y}),\, \phi_{k-1}(\bar{x}, \bar{y}) \vee \phi_k(\bar{x}, \bar{y}).\] Consider the $2$-colouring $\chi_2 : \binom{M}{\bar{a}} \to 2$ externally defined by $\bar{e} \in N \succeq M$ and the formulae $\phi_0(\bar{x}, \bar{y}) \vee \cdots \vee \phi_{k-1}(\bar{x}, \bar{y})$ and $\phi_k(\bar{x}, \bar{y})$.
			
        We will apply \Cref{Ramseyforage}, with \[\Phi = \{\phi_0(\bar{x}, \bar{y}) \vee \cdots \vee \phi_{k-1}(\bar{x}, \bar{y}), \phi_k(\bar{x}, \bar{y})\}.\] As $M \to_{\ext(\Phi)} (B)^{\bar{a}}_2$, by \Cref{Ramseyforage}, there is $C \in \Age(M)$ with $C \to_{\chi_2} (B)^{\bar{a}}_2$. 
        By the induction assumption, $M$ has a $\chi_1$-monochromatic copy $C'$ of $C$. If the $\chi_1$-monochromatic colour of $C'$ is one of $0, \cdots, k-2$, then $C'$ is monochromatic in $\chi$, and as $C'$ contains a copy $B'$ of $B$, $B'$ is monochromatic in $\chi$, so we are done. We now consider the case where $C'$ has $\chi_1$-colour $k-1$. In this case, we have that $C'$ has colours $k-1$ and $k$ in $\chi$.
			
        In the $2$-colouring $\chi_2$, $C'$ has a $\chi_2$-monochromatic copy $B'$ of $B$. If $B'$ has the colour externally defined by $\phi_k(\bar{x}, \bar{y})$, then $B'$ is monochromatic in $\chi$. If $B'$ has the colour externally defined by $\phi_0(\bar{x}, \bar{y}) \vee \cdots \vee \phi_{k-1}(\bar{x}, \bar{y})$ in $\chi_2$, then as $C'$ has colours $k-1, k$ in $\chi$, we have that $B'$ has the colour in $\chi$ externally defined by $\phi_{k-1}(\bar{x}, \bar{y})$, so $B'$ is monochromatic in $\chi$.
    \end{proof}
		
    \begin{lem} \label{boolcomb}
        Let $M$ be an ultrahomogeneous structure. Let $\bar{a} \in \Age(M)$. Let $\Phi$ be a set of formulae (where free variables in $\bar{y}$ correspond to $\bar{a}$).

        Suppose that $M$ has the $\bar{a}$-RP for all $2$-colourings externally defined by a formula from $\Phi$. Then $M$ has the $\bar{a}$-RP for all $2$-colourings externally defined by a formula consisting of a Boolean combination of formulae from $\Phi$.
    \end{lem}
    \begin{proof}
        Let $\Psi$ be the set consisting of the formulae which are Boolean combinations of formulae in $\Phi$, where we assume all Boolean combinations are expressed in terms of $\neg$ and $\wedge$. For $\psi \in \Psi$, we inductively define the $\Phi$-complexity $l(\psi)$ of $\psi$ as follows:
        \[
        l(\psi) =
        \begin{cases}
            0 & \psi \in \Phi \\
            l(\psi_1) + 1 & \psi = \neg\psi_1 \\
            \max(l(\psi_1), l(\psi_2)) + 1 & \psi = \psi_1 \wedge \psi_2
        \end{cases}
        \]

        For $s \in \N$, let $\Psi_s = \{\psi \in \Psi : l(\psi) \leq s\}$. Using induction on $s \in \N$, we will show that $M$ has the $\bar{a}$-RP for each $2$-colouring externally defined by a formula in $\Psi_s$.
			
        The base case is given by the assumption that $M$ has the $\bar{a}$-RP for all $2$-colourings externally defined by a formula from $\Phi$, as $\Psi_0 = \Phi$.
			
        Assume that $M$ has the $\bar{a}$-RP for $2$-colourings externally defined by a formula from $\Psi_s$, and let $\chi : \binom{M}{\bar{a}} \to 2$ be a $2$-colouring externally defined by $\bar{e} \in N \succeq M$ and a formula $\psi(\bar{x}, \bar{y}) \in \Psi_{s+1}$ of complexity $s + 1$. We must have one of the following:
        \begin{enumerate}
            \item $\psi(\bar{x}, \bar{y}) = \neg \psi_1(\bar{x}, \bar{y})$, where $\psi_1(\bar{x}, \bar{y})$ is of complexity $s$;
            \item $\psi(\bar{x}, \bar{y}) = \psi_1(\bar{x}, \bar{y}) \,\wedge\, \psi_2(\bar{x}, \bar{y})$, where $\psi_1(\bar{x}, \bar{y}), \psi_2(\bar{x}, \bar{y})$ are of complexity $\leq s$ and one of the two formulae is of complexity exactly $s$.
        \end{enumerate}
			
        For the first case, take the colouring $\chi_1 : \binom{M}{\bar{a}} \to 2$ externally defined by $\bar{e} \in N \succeq M$ and $\psi_1(\bar{x}, \bar{y})$. Then $\chi = 1 - \chi_1$, and so by the induction assumption $M$ has the $\bar{a}$-RP for $\chi$.
			
        In the second case, as $\psi_1(\bar{x}, \bar{y})$, $\psi_2(\bar{x}, \bar{y})$ are of complexity $\leq s$, by the induction assumption $M$ has the $\bar{a}$-RP for colourings externally defined by $\psi_1(\bar{x}, \bar{y})$ and for colourings externally defined by $\psi_2(\bar{x}, \bar{y})$. Let $\chi_1 : \binom{M}{\bar{a}} \to 2$ be the colouring externally defined by $\bar{e} \in N \succeq M$ and $\psi_1(\bar{x}, \bar{y})$, and let $\chi_2 : \binom{M}{\bar{a}} \to 2$ be the colouring externally defined by $\bar{e} \in N \succeq M$ and $\psi_2(\bar{x}, \bar{y})$. Take $B \in \Age(M)$. Applying \Cref{Ramseyforage} to the set $\{\psi_2(\bar{x}, \bar{y})\}$, there is $C \in \Age(M)$ with $C \to_{\chi_2} (B)^{\bar{a}}_2$, i.e.\ for any copy $C'$ of $C$ in $M$, $C'$ contains a $\chi_2$-monochromatic copy of $B$. As $M$ has the $\bar{a}$-RP for $\chi_1$, there is a $\chi_1$-monochromatic copy $C' \sub M$ of $C$. $C'$ contains a $\chi_2$-monochromatic copy $B'$ of $B$, and so $B'$ is monochromatic in $\chi_1$ and $\chi_2$. Therefore $B'$ is monochromatic in $\chi$.
    \end{proof}

\subsection{Simplifications via quantifier elimination}

    \begin{defn}
        Let $M$ be a structure. Let $\bar{a} \in \Age(M)$.
			
        If $\chi : \binom{M}{\bar{a}} \to 2$ is externally defined by an atomic formula $\phi(\bar{x}, \bar{y})$, we say that $\chi$ is \emph{externally atomic-definable}.
    \end{defn}

    \begin{lem} \label{EDRPatomic}
        Let $M$ be an ultrahomogeneous structure. Suppose $\Th(M)$ has QE. Let $\bar{a} \in \Age(M)$. Suppose $M$ has the $\bar{a}$-RP for externally atomic-definable $2$-colourings. Then M has the $\bar{a}$-RP for all externally definable $2$-colourings.
    \end{lem}
    \begin{proof}
        Apply \Cref{boolcomb} to the set of atomic formulae. 
    \end{proof}
    \begin{rem} \label{equalityflastrivial}
        Assume $M$ is ultrahomogeneous. We consider the case of equality. The $2$-colourings externally defined by the atomic formula $\phi(\bar{x}, \bar{y}) \equiv x_0 = x_1$ or by $\phi(\bar{x}, \bar{y}) \equiv y_0 = y_1$ are constant, so immediately we have that $M$ has the $\bar{a}$-RP for these colourings. However, it is possible for $M$ to not have the point-Ramsey property for $2$-colourings externally defined by the atomic formula $\phi(\bar{x}, \bar{y}) \equiv x = y$. For example, let $\mc{L}$ consist of a single unary predicate $R$, which we think of as the ``red points" of a set, and let $M$ be the \Fr limit of the class of finite $\mc{L}$-structures $D$ with $|R^D| \leq 2$, i.e.\ with at most two red points. Let $a \in M$ be a red point, and let $B \in \Age(M)$ consist of two red points. Then the $2$-colouring of $\binom{M}{a}$ externally defined by $a \in M$ and the formula $\phi(x, y) \equiv x = y$ has no monochromatic copy of $B$.

        However, if $\Age(M)$ has the strong JEP (the JEP with the additional condition that the images of the two embeddings are disjoint), then $M$ has the point-Ramsey property for $2$-colourings externally defined by $e \in M$ and $x = y$: any copy of $B$ which does not contain $e$ is monochromatic. Thus, if $\Age(M)$ has the strong JEP, $M$ automatically has the $\bar{a}$-RP for all $2$-colourings externally defined by atomic formulae involving equality.
    \end{rem}
	
    \begin{lem} \label{EDRPt}
        Let $M$ be an ultrahomogeneous $\mc{L}$-structure, where $\mc{L}$ has maximum arity $t \geq 2$. Suppose $\Th(M)$ has QE. Suppose $M$ has the externally definable $\bar{a}$-Ramsey property for all $\bar{a} \in \Age(M)$ with $|\bar{a}| < t$. Then $M$ has the externally definable Ramsey property.
    \end{lem}
    \begin{proof}
        By \Cref{EDRPkiff2} and \Cref{EDRPatomic}, it suffices to consider externally atomic-definable $2$-colourings. Let $\bar{a} \in \Age(M)$, and let $\chi : \binom{M}{\bar{a}} \to 2$ be externally defined by some $\bar{e} \in N \succeq M$ and an atomic formula $\phi(\bar{x}, \bar{y})$. Take subsequences $(x_{r(i)}), (y_{s(i)})$ of $\bar{x}, \bar{y}$ respectively, consisting of the free variables that actually occur in $\phi(\bar{x}, \bar{y})$. Let $u$ be the length of the subsequence $(x_{r(i)})$, and let $v$ be the length of the subsequence $(y_{s(i)})$. 

        As $\phi(\bar{x}, \bar{y})$ is atomic, we have that $u + v \leq t$. (Note that here, as we have assumed $t \geq 2$, this is also true for the case of atomic formulae involving the equality relation.) If $u = 0$, then $\chi$ is a constant colouring, so certainly has a monochromatic copy of each $B \in \Age(M)$. 
        
        Now consider the case $u \geq 1$. We then have $v < t$. Let $\bar{e}_r = (e_{r(i)})_{i < u}$, $\bar{a}_s = (a_{s(i)})_{i < v}$, and let $\chi_s : \binom{M}{\bar{a}_s} \to 2$ be externally defined by $\bar{e}_r \in N \succeq M$ and $\phi((x_{r(i)}), (y_{s(i)}))$.

        Let $B \in \Age(M)$. Then, by assumption, $M$ contains a copy $B'$ of $B$ with $\binom{B'}{\bar{a}_s}$ monochromatic in $\chi_s$. But also for any $\bar{a}'_s \in \binom{M}{\bar{a}_s}$ and for any extension $\bar{a}'$ of $\bar{a}_s$ to a copy of $\bar{a}$, we have that $\chi(\bar{a}') = \chi^{}_s(\bar{a}'_s)$. So therefore $\binom{B}{\bar{a}}$ is monochromatic in $\chi$.
    \end{proof}
	
    We combine \Cref{EDRPkiff2,EDRPatomic,EDRPt} to obtain the following, which we shall frequently use in examples:
		
    \begin{prop} \label{EDRPsimplecheck}
        Let $M$ be an ultrahomogeneous $\mc{L}$-structure, where $\mc{L}$ has maximum arity $t \geq 2$. Suppose $\Th(M)$ has QE. Suppose, for $\bar{a} \in \Age(M)$ with $|\bar{a}| < t$, that $M$ has the $\bar{a}$-RP for externally atomic-definable $2$-colourings. Then $M$ has the externally definable Ramsey property.
    \end{prop}

    (This resembles \cite[Proposition 3.17]{KLM}.)
 
\section{The fixed points on type spaces property (FPT)} \label{FPTsec}

    The below lemma is folklore, and its proof follows via a straightforward compactness argument.
 
    \begin{lem} \label{typeflow}
        Let $M$ be an $\mc{L}$-structure, and let $G = \Aut(M)$ with the pointwise convergence topology. Then $G$ acts continuously on the Stone spaces $S_n(M)$, with the action given by \[g \cdot p(\bar{x}) = \{\phi(\bar{x}, g\bar{m}) : \phi(\bar{x}, \bar{m}) \in p(\bar{x})\}.\] That is, $G \curvearrowright S_n(M)$ with the action defined above is a $G$-flow.
    \end{lem}
		
    Note that we define the action of $G$ on $\mc{L}(M)$-formulae as $g \cdot \phi(\bar{x}, \bar{m}) = \phi(\bar{x}, g\bar{m})$.		
    
    \begin{defn}
        Let $M$ be an $\mc{L}$-structure with automorphism group $G = \Aut(M)$. We say that $M$ has the \emph{fixed points on type spaces property} (FPT) if every subflow of $G \curvearrowright S_n(M)$, $n \geq 1$, has a fixed point.

        We also define a more restricted version of this property for particular $n$. Let $n \geq 1$. We say that $M$ has \emph{FPT$_n$} if every subflow of $G \curvearrowright S_n(M)$ has a fixed point.
    \end{defn}
    Note that FPT is equivalent to every orbit closure $\ov{G \cdot p(\bar{x})}$ in $S_n(M)$ having a fixed point.

    The following lemma will not be used in later sections, but we consider it to be of independent interest.

    \begin{lem} \label{FPTn+1impliesn}
        Let $M$ be a structure with $FPT_{n+1}$. Then $M$ has $FPT_n$.
    \end{lem}
    \begin{proof}
        Let $\pi : S_{n+1}(M) \to S_n(M)$ be the restriction map to the first $n$ variables, i.e.\ for a type $p(\bar{x}, y) \in S_{n+1}(M)$, the type $\pi(p(\bar{x}, y))$ consists of the formulae of $p(\bar{x}, y)$ with free variables from $\bar{x}$. 
        Checking that $\pi(p(\bar{x}, y))$ is indeed a type is straightforward, and it is also immediate that $\pi$ is continuous and surjective.

        Let $G = \Aut(M)$. Let $p(\bar{x}) \in S_n(M)$. To show that $M$ has $FPT_n$, it suffices to show that $\ov{G \cdot p(\bar{x})}$ contains a $G$-invariant type. Take $\tilde{p}(\bar{x}, y) \in \pi^{-1}(p(\bar{x}))$. As $M$ has $FPT_{n+1}$, there exists a $G$-invariant type $\tilde{q}(\bar{x}, y) \in \ov{G \cdot \tilde{p}(\bar{x}, y)}$. Let $q(\bar{x}) = \pi(\tilde{q}(\bar{x}, y))$. Then we have that $q(\bar{x}) \in \pi(\ov{G \cdot \tilde{p}(\bar{x}, y)})$.

        Recall that for a continuous function $f : X \to Y$ where $X$ is compact and $Y$ is Hausdorff, for $A \sub X$ we have that $f(\ov{A}) = \ov{f(A)}$. 
        
        Therefore $\pi(\ov{G \cdot \tilde{p}(\bar{x}, y)}) = \ov{\pi(G \cdot \tilde{p}(\bar{x}, y))}$. It is easy to see that for $g \in G$ and $p'(\bar{x}, y) \in S_{n+1}(M)$, we have that $\pi(g \cdot p'(\bar{x}, y)) = g \cdot \pi(p'(\bar{x}, y))$. Hence $\ov{\pi(G \cdot \tilde{p}(\bar{x}, y))} = \ov{G \cdot p(\bar{x})}$. So $q(\bar{x}) \in \ov{G \cdot p(\bar{x})}$. As $\tilde{q}(\bar{x}, y)$ is $G$-invariant, so is $q(\bar{x})$, and thus we are done. 
    \end{proof}
		
\section{FPT and the externally definable Ramsey property} \label{mainthmsec}
    \begin{thm} \label{FPTiffEDRP}
        Let $M$ be an ultrahomogeneous structure.
			
        Then:
        \begin{enumerate}
            \item if $M$ has FPT, then $M$ has the externally definable Ramsey property;
            \item if $\Age(M)$ is countable, then $M$ has the externally definable Ramsey property iff $M$ has FPT.
        \end{enumerate}
    \end{thm}
    \begin{proof}
        We write $G = \Aut(M)$.
        
        (1): Let $\bar{a} \in \Age(M)$. We may assume $\bar{a} \sub M$. By \Cref{EDRPkiff2}, it suffices to show the $\bar{a}$-RP for externally definable $2$-colourings. Let $\chi : \binom{M}{\bar{a}} \to 2$ be externally defined by $\bar{e} \in N \succeq M$ and $\phi(\bar{x}, \bar{y})$, where $\phi(\bar{x}, \bar{y})$ has free variables among $\bar{x}, \bar{y}$, but including all of $\bar{x}$. We show that for $B \in \Age(M)$, there exists a $\chi$-monochromatic copy $B' \sub M$ of $B$.
			
        Let $p(\bar{x}) = \tp_N(\bar{e}/M)$. As $M$ has FPT, there is a $G$-invariant type $q(\bar{x}) \in \ov{Gp(\bar{x})}$. Thus:
			
        \begin{enumerate}
            \item[(i.)] for all $g \in G$ and for all $f(\bar{x}, \bar{m}) \in q(\bar{x})$, we have $f(\bar{x}, g\bar{m}) \in q(\bar{x})$;
            \item[(ii.)] for all $f(\bar{x}, \bar{m}) \in q(\bar{x})$, there is $g \in G$ such that $f(\bar{x}, g\bar{m}) \in p(\bar{x})$.
        \end{enumerate}
			
        Let $\phi_0(\bar{x}, \bar{y}) = \phi(\bar{x}, \bar{y})$ and $\phi_1(\bar{x}, \bar{y}) = \neg\phi(\bar{x}, \bar{y})$. There is $i \in \{0, 1\}$ with $\phi_i(\bar{x}, \bar{a}) \in q(\bar{x})$. By (i), using the ultrahomogeneity of $M$, for $\bar{a}' \in \binom{M}{\bar{a}}$, we have $\phi_i(\bar{x}, \bar{a}') \in q(\bar{x})$. Take $B \fin M$. Then $\bigwedge_{\, \bar{a}' \in \binom{B}{\bar{a}}} \phi_i(\bar{x}, \bar{a}') \in q(\bar{x})$, and so by (ii.)\ there is $g \in G$ such that $\bigwedge_{\, \bar{a}' \in \binom{B}{\bar{a}}} \phi_i(\bar{x}, g\bar{a}') \in p(\bar{x})$. Thus $\binom{gB}{\bar{a}}$ is monochromatic in $\chi$.

        (2): Take $p(\bar{x}) \in S_n(M)$. We will show that there is a $G$-invariant type $q(\bar{x}) \in \ov{Gp(\bar{x})}$.
        
        Let $N \succeq M$ and $\bar{e} \in N$ be such that $\tp_N(\bar{e}/M) = p(\bar{x})$. Throughout this proof, when we specify externally defined colourings, the variables $\bar{x}$ will correspond to $\bar{e}$.
        
        First, note that for any $\mc{L}$-formula $\phi(\bar{z})$, we may partition $\bar{z}$ into $\bar{z} \cap \bar{x}$ and $\bar{y} = \bar{z} \setminus \bar{x}$, and thereby consider $\phi(\bar{z})$ as a formula $\phi(\bar{x}, \bar{y})$. For each $\mc{L}$-formula $\phi(\bar{x}, \bar{y})$ and each $\bar{a} \in \Age(M)$ with $|\bar{a}| = |\bar{y}|$, let $\chi_{\phi, \bar{a}} : \binom{M}{\bar{a}} \to 2$ denote the colouring externally defined by $\bar{e} \in N \succeq M$ and $\phi(\bar{x}, \bar{y})$. Note that for $\bar{a}, \bar{b} \in \Age(M)$ with $\bar{a} \cong \bar{b}$, $\chi_{\phi, \bar{a}} = \chi_{\phi, \bar{b}}$.
			
        As $\mc{L}$ and $\Age(M)$ are countable, there are countably many $\chi_{\phi, \bar{a}}$, which we enumerate as $\chi_i = \chi_{\phi_i, \bar{a}_i}$, $i \in \N$.
			
        \textbf{Claim 1:} Let $B \in \Age(M)$, $r \in \N$. Then there exists a copy $B' \sub M$ of $B$ such that $B'$ is $\chi_i$-monochromatic for all $i < r$.
			
        \begin{subproof}[Proof of claim:]
            As $M$ has externally definable $\bar{a}_0$-Ramsey for $\chi_0$, by \Cref{Ramseyforage} there is $C_1 \in \Age(M)$ with $C_1 \to_{\chi_0} (B)^{\bar{a}_0}_2$. Similarly, by induction we may take $C_i \in \Age(M)$, $1 \leq i \leq r$, such that $C_i \to_{\chi_{i-1}} (C_{i-1})^{\bar{a}_{i-1}}_2$, and we may take $C_r \sub M$. Then $C_r$ contains a copy $B'$ of $B$ with $\binom{B'}{\bar{a}_i}$ monochromatic in $\chi_i$ for $0 \leq i < r$.  
        \end{subproof}

        \textbf{Claim 2:} There exist colours $(\lambda_r)_{r < \omega}$ such that, for $B \in \Age(M)$, $r < \omega$, there exists a copy $B' \sub M$ of $B$ with $B'$ having $\chi_i$-colour $\lambda_i$ for $i < r$.
   
        \begin{subproof}[Proof of claim:]
            Use JEP and EP to find an increasing chain $B_0 \sub B_1 \sub \cdots$ of finite substructures of $M$ such that each $B \in \Age(M)$ embeds in some $B_r$. We define a tree $(T, \leq)$: for $r < \omega$ and $\rho \in 2^r$, add $(B_r, \rho)$ to $T$ if there exists a copy $B'_r \sub M$ of $B_r$ with $B'_r$ having $\chi_i$-colour $\rho_i$ for $i < r$, and define $(B_r, \rho) \leq (B_s, \sigma)$ if $r \leq s$ and $\rho$ is an initial segment of $\sigma$. By Claim 1 and K\H{o}nig's lemma, the tree $T$ has an infinite branch.
        \end{subproof} 
			
        Let ${(\lambda_r)}_{r < \omega}$ be given by the above Claim 2. Define $q(\bar{x})$ as follows. For $r < \omega$, let \[\psi_r(\bar{x}, \bar{y}) = 
        \begin{cases}
            \phi_r(\bar{x}, \bar{y}) & \lambda_r = 0 \\
            \neg\phi_r(\bar{x}, \bar{y}) & \lambda_r = 1
        \end{cases}\]
        and add $\psi_r(\bar{x}, \bar{a}'_r)$ to $q(\bar{x})$ for all $\bar{a}'_r \in \binom{M}{\bar{a}_r}$.
        It is immediate that $q(\bar{x})$ is $G$-invariant. To show that $q(\bar{x})$ is a complete type, it suffices to show that $q(\bar{x}) \cup \Th_{\mc{L}(M)}(M)$ is consistent. Take $s, t \in \N$. For $i < s, j < t$, take $\bar{a}_{i, j} \in \binom{M}{\bar{a}_i}$. We show that \[\{\psi_i(\bar{x}, \bar{a}_{i, j}) : i < s, j < t\} \cup \Th_{\mc{L}(M)}(M)\] is consistent, and this suffices by compactness. Take finite $B \sub M$ with $\bar{a}_{i, j} \sub B$ for $i < s, j < t$. Then, by Claim 2, there is $B' \sub M$ and an isomorphism $f : B \to B'$ with $\chi_i |_{B'} = \{\lambda_i\} \text{ for } i < s$. Using ultrahomogeneity of $M$, there is $g \in \Aut(M)$ extending $f$. As $\chi_i |_{B'} = \{\lambda_i\}$ for $i < s$, we have $N \models \psi_i(\bar{e}, \bar{a}'_i) \text{ for } i < s \text{ and } \bar{a}'_i \in \binom{B'}{\bar{a}_i}$, and hence $N \models \psi_i (\bar{e}, g\bar{a}_{i, j}) \text{ for } i < s, j < t$, which gives the desired consistency.

        It remains to show that $q(\bar{x}) \in \ov{Gp(\bar{x})}$: this is equivalent to the statement that for $f(\bar{x}, \bar{m}) \in q(\bar{x})$, there is $g \in G$ such that $f(\bar{x}, g\bar{m}) \in p(\bar{x})$, and this follows immediately from Claim 2 and the definition of $q(\bar{x})$. 
    \end{proof}

    We note that, for the case of $\aleph_0$-saturated \Fr structures, \Cref{FPTiffEDRP} can be deduced from \cite[Theorem 3.14]{KLM} (personal communication, K.\ Krupiński). The proof of \Cref{FPTiffEDRP} given here uses elementary techniques, and applies to \Fr structures generally (or indeed ultrahomogeneous structures with countable age). Several of the examples in the following section are not $\aleph_0$-saturated, so \cite[Theorem 3.14]{KLM} cannot be applied, and here \Cref{FPTiffEDRP} becomes relevant. In \Cref{infinite products}, we give a general recipe for constructing such examples.

\section{Examples} \label{exsec}

    We now consider a range of examples. We evaluate whether each example has the externally definable Ramsey property. We shall frequently use \Cref{EDRPsimplecheck} to verify the externally definable Ramsey property.

    The salient properties of the examples in this section are summarised in the below table.

    \medskip

    \begin{adjustbox}{center}
        \begin{tabular}{|l|c|c|c|c|c|c|c|}
            \hline
            Structure & EDRP & FPT & RP $< t$ & homog.\ & QE & $\aleph_0$-cat.\ & $\aleph_0$-sat.\ \\
            \hline
            random graph & Y & Y & Y & Y & Y & Y & Y \\
            equiv.\ rel.\ $M_{\omega, j}$, $j < \omega$ & Y & Y & N & Y & Y & Y & Y \\
            $S(2)$ & N & N & N & Y & Y & Y & Y \\
            sat.\ rainbow graph & Y & Y & Y & Y & Y & N & Y \\
            unsat.\ rainbow graph & Y & Y & Y & Y & Y & N & N \\
            sat.\ inf.\ ref.\ eq.\ rels & Y & Y & Y & Y & Y & N & Y \\
            unsat.\ inf.\ ref.\ eq.\ rels & Y & Y & Y & Y & Y & N & N \\
            $2$ disj.\ random graphs & Y & N & Y & N & N & Y & Y \\
            \hline
        \end{tabular}
    \end{adjustbox}
    \hfill \break
    \textbf{Key:} 
    \begin{description}
        \item[EDRP] externally definable Ramsey property; 
        \item[RP $< t$] $\bar{a}$-Ramsey property for $|\bar{a}| < t$, where $t$ is the maximum arity of $\mc{L}$; 
        \item [homog.] ultrahomogeneous.
    \end{description}
    
    \subsection{Finite binary languages: the point-Ramsey property}
        \begin{prop}
            Let $\mc{L}$ be a language consisting of finitely many binary relation symbols. Let $M$ be a \Fr $\mc{L}$-structure with the point-Ramsey property. Then $M$ has the externally definable Ramsey property.
        \end{prop}
        \begin{proof}
            By \Cref{finLQE}, $\Th(M)$ has QE. Thus, by \Cref{EDRPsimplecheck}, we need only check the point-Ramsey property for $2$-colourings of the points of $M$ which are externally defined by atomic formulae, and these form a subset of all $2$-colourings of points of $M$. As $M$ has the point-Ramsey property, we therefore have the externally definable Ramsey property.
	\end{proof}
        \begin{prop}
            The random graph has the point-Ramsey property.
        \end{prop}
	\begin{proof}
            Denote the random graph by $M$. We show that $M$ has the pigeonhole property (see \Cref{pigeonindivdef}), a stronger property than the point-Ramsey property. The following is from \cite[Proposition 1.3.3]{Cam97} -- we include it due to its brevity. Recall that the random graph has the \emph{witness property}:
		\begin{center}
                $(\ast)$ for finite disjoint $U, V \sub M$, there exists $m \in M \setminus (U \cup V)$ such that $m$ is adjacent to every vertex of $U$ and no vertex of $V$.
            \end{center}
            Also, any countable graph with the witness property is isomorphic to the random graph by a back-and-forth argument.
				
            Suppose, for a contradiction, that $M$ has a partition $M = M_1 \cup M_2$ with $M_1, M_2 \not\cong M$. Then, for $i = 1, 2$, there exist $U_i, V_i \sub M_i$ with no witness in $M_i$, i.e.\ there is no $m \in M_i$ with $m$ adjacent to all of $U_i$ and none of $V_i$. But then $U_1 \cup U_2, V_1 \cup V_2$ have no witness in $M$, contradiction.
	\end{proof}
        One can also quickly prove the point-Ramsey property for the random graph via the lexicographic product: see \cite[Theorem 12.1]{Pro13}.
		
        There are many other \Fr structures in finite binary relational languages with the point-Ramsey property, and therefore the externally definable Ramsey property. For example:
        \begin{itemize}
            \item the random tournament: this structure has the pigeonhole property by an argument as for the random graph, and we can also prove the point-Ramsey property via the lexicographic product;
            \item the $K_n$-free random graph ($n \geq 3$): this follows via Folkman's theorem, the proof of which adapts the lexicographic product proof for the random graph (\cite{Fol70}; a clear exposition following \cite{Kom86} is \cite[Theorem 12.3]{Pro13}); alternatively, we may use the \NR theorem as in \Cref{freeamalgexs};
            \item the random poset: for $(B, \leq)$ a finite poset, the map $b \mapsto \{b' \in B : b' \leq b\}$ is an embedding $(B, \leq) \to (\mc{P}(B), \sub)$. A straightforward application of the Hales-Jewett theorem for the alphabet $\{0, 1\}$ and $|B|$-parameter sets gives the point-Ramsey property (see, for example, \cite[Proposition 4.2.5]{Pro13}).
        \end{itemize}

    \subsection{Equivalence relations} \label{singleeqrel}
        Let $i, j \in \N \cup \{\omega\}$, with at least one of $i, j$ equal to $\omega$. Let $\mc{L} = \{\sim\}$ with $\sim$ binary. Let $M_{i, j}$ be an $\mc{L}$-structure with countably infinite domain, where $\sim^M$ is an equivalence relation with $i$ equivalence classes of size $j$. Each $M_{i, j}$ is ultrahomogeneous, and so by \Cref{finLQE}, $\Th(M_{i, j})$ has QE.

        \begin{prop}
            Let $i, j \geq 2$.
            \begin{enumerate}
                \item If $i = \omega$, $j < \omega$, then $M_{i, j}$ has the externally definable Ramsey property but does not have the point-Ramsey property.
                \item If $i < \omega, j = \omega$, then $M_{i, j}$ does not have the externally definable Ramsey property (and hence does not have the point-Ramsey property).
                \item If $i = j = \omega$, then $M_{i, j}$ has the point-Ramsey property, and hence the externally definable Ramsey property.
            \end{enumerate}
        \end{prop}
        \begin{proof}
            (1): By \Cref{EDRPsimplecheck} and \Cref{equalityflastrivial} (note that $\Age(M_{i, j})$ has the strong JEP), we need only check the externally definable RP for $|\bar{a}| = 1$ and $\phi(x, y) \equiv x \sim y$. Let $\chi : \binom{M_{i, j}}{\bar{a}} \to 2$ be externally defined by some $e \in N \succeq M_{i, j}$ and $\phi(x, y)$. As all equivalence classes in $N$ are of size $\leq j$, we have $|\chi^{-1}(0)| \leq j$, and so using the EP of $M_{i, j}$, for $B \in \Age(M_{i, j})$ we find a copy $B' \sub M_{i, j}$ of $B$ of colour $1$ in $\chi$, and hence $M_{i, j}$ has the externally definable RP.
			
            To see that $M_{i, j}$ does not have the point-Ramsey property: take a $j$-colouring of the points of $M_{i, j}$ where for each equivalence class $E$ no two of the $j$ points of $E$ have the same colour.
                
            (2): Let $e \in M_{i, j}$. Consider the colouring $\chi : \binom{M_{i, j}}{a} \to 2$ externally defined by $e \in M_{i, j}$ and $\phi(x, y) \equiv x \sim y$. We have that $\chi^{-1}(0)$ is the equivalence class of $e$ in $M_{i, j}$ and $\chi^{-1}(1)$ is the union of the remaining $i - 1$ equivalence classes. So taking $B$ to be $i$ inequivalent points, there is no $\chi$-monochromatic copy of $B$ in $M_{i, j}$.
            
            (3): It is straightforward to see that $M_{i, j}$ is indivisible, and hence by \Cref{pigeonindivlem} and \Cref{ageindivpointRamsey} we see that $M_{i, j}$ has the point-Ramsey property.
        \end{proof}
        
    \subsection{Free amalgamation classes via \NR} \label{freeamalgexs}
        \begin{defn}
            Let $\mc{K}$ be a \Fr class of $\mc{L}$-structures with strong amalgamation. Let $\mc{L}_\leq = \mc{L} \cup \{\leq\}$, where $\leq$ is binary, and let $\mc{K}_\leq$ be the class of $\mc{L}_\leq$-structures $(A, \leq_A)$, where $A \in \mc{K}$ and $\leq_A$ is any linear order on $A$. It is easy to see that $\mc{K}_\leq$ is a \Fr class with strong amalgamation. We call $\mc{K}_\leq$ the \emph{order expansion} of $\mc{K}$.
        \end{defn}
        \begin{thm}[\NRn, \cite{NR77a}, \cite{NR77b}] \label{NRfreeamalg}
            Let $\mc{K}$ be a free amalgamation class in a countable relational language. Then the order expansion $\mc{K}_\leq$ has the Ramsey property.
        \end{thm}
        \begin{prop} \label{EDRPbyNR}
            Let $n \geq 2$, and let $\mc{L} = \{R_i : i \in I\}$ with each $R_i$ of arity $n$. Let $M$ be a \Fr $\mc{L}$-structure with age $\mc{K}$ such that:
            \begin{itemize}
                \item each $R_i^M$ is symmetric;
                \item $\mc{K}$ has free amalgamation;
                \item $\Th(M)$ has QE.
            \end{itemize}
            Then $M$ has the externally definable Ramsey property.
        \end{prop}
        \begin{proof}
            (Similar to {\cite[Example 5.6]{KLM}}.) Let $\mc{K}_\leq$ be the order expansion of $\mc{K}$ with \Fr limit $M_\leq$ (with $M$ the $\mc{L}$-reduct of $M_\leq$). By \Cref{NRfreeamalg}, we see that $M_\leq$ has the RP. To show that $M$ has the externally definable RP, by \Cref{EDRPsimplecheck} it suffices to check the $\bar{a}$-RP for externally atomic-definable $2$-colourings, where $|\bar{a}| < n$. Let $i \in I$ and $\bar{a} \in \mc{K}$, $|\bar{a}| < n$. Let $\chi : \binom{M}{\bar{a}} \to 2$ be externally defined by $\bar{e} \in N \succeq M$ (where $|\bar{e}| + |\bar{a}| = n$) and $\phi(\bar{x}, \bar{y}) \equiv R_i(\bar{x}, \bar{y})$. Let $B \in \mc{K}$.

            Let $(\bar{a}, \leq_A), (B, \leq_B) \in \mc{K}_\leq$. We induce a colouring $\chi_{\text{exp}} : \binom{M_\leq}{(\bar{a}, \leq_A)} \to 2$ by $\chi_{\text{exp}} (\bar{a}', \leq_{A'}) = \chi(\bar{a}')$. As $M_\leq$ has the Ramsey property, there is a $\chi_{\text{exp}}$-monochromatic copy $(B', \leq_{B'}) \sub M_\leq$ of $(B, \leq_B)$. Take $\bar{a}' \in \binom{B'}{\bar{a}}$. There is a permutation $\sigma(\bar{a}')$ of $\bar{a}'$ with $(\sigma(\bar{a}'), \leq_{B'}) \cong (\bar{a}, \leq_A)$, and as $R_i^M$ is symmetric, we have that $\chi(\bar{a}') = \chi(\sigma(\bar{a}')) = \chi_{\text{exp}} ((\sigma(\bar{a}'), \leq_{B'}))$. Thus $B'$ is $\chi$-monochromatic.
        \end{proof}
        
        If $n = 2$, then we only need to check externally definable Ramsey for $|\bar{a}| = 1$, and so we do not require that each $R_i^M$ is symmetric. So we have the following special case of the previous proposition:
        \begin{prop} \label{EDRPbyNRbinary}
            Let $\mc{L}$ be a finite binary relational language. Let $M$ be a \Fr $\mc{L}$-structure whose age has free amalgamation. Then $M$ has the externally definable Ramsey property.
        \end{prop}

        We apply \Cref{EDRPbyNR,EDRPbyNRbinary} to see that the random graph, random $K_n$-free graph and random $n$-hypergraph have the externally definable Ramsey property.
    
    \subsection{The directed graphs \texorpdfstring{$S(2)$}{S(2)} and \texorpdfstring{$S(3)$}{S(3)}}
    
        The \emph{dense local order} $S(2)$ is defined as follows. Let $\mc{L} = \{R\}$ with $R$ binary. Let $\mathbb{T}$ denote the unit circle in the complex plane, and let $S(2)$ be the $\mc{L}$-structure whose domain consists of the points of $\mathbb{T}$ with rational argument, and where $aRb$ iff the angle subtended at the origin by the directed arc from $a$ to $b$ is less than $\pi$. As the domain of $S(2)$ excludes antipodal points, $S(2)$ is a tournament.


        We have that $S(2)$ is ultrahomogeneous - see \cite{Lac84} for further details.

        \begin{prop}
            The dense local order $S(2)$ does not have the externally definable Ramsey property.
        \end{prop}
        \begin{proof}
            Take $e \in S(2)$, and let $\phi(x, y) \equiv x R y$. Let $\chi$ be the vertex-colouring of $S(2)$ externally defined by $e \in S(2)$ and $\phi(x, y)$. It is then immediate that $\chi$ does not have a monochromatic copy of $\overrightarrow{C}_3 \in \Age(S(2))$.
        \end{proof}

        The digraph $S(3)$ is defined in a similar manner to $S(2)$. It has the same domain, but for $S(3)$ we define the digraph structure by: $aRb$ iff the angle subtended at the origin by the directed arc from $a$ to $b$ is less than $\frac{2\pi}{3}$. Note that $S(3)$ is not a tournament. We have that $S(3)$ is ultrahomogeneous -- see \cite{Che98}.

        \begin{prop}
            The digraph $S(3)$ does not have the externally definable Ramsey property.
        \end{prop}
        \begin{proof}
            Let $\overrightarrow{C}_4$ denote the directed $4$-cycle. Then it is straightforward to check that $\overrightarrow{C}_4 \in \Age(S(3))$. Let $e \in S(3)$, let $\phi(x, y) \equiv x R y$, and let $\chi$ be the vertex-colouring externally defined by $e \in S(3)$ and $\phi(x, y)$. Then for any copy of $\overrightarrow{C}_4$ in $S(3)$, $e$ has an out-edge to one vertex of the copy and an in-edge from another vertex of the copy. Therefore $\chi$ has no monochromatic copy of $\overrightarrow{C}_4$.
        \end{proof}
        
    \subsection{The rainbow graph, saturated version}
        \begin{defn} \label{omegacolouredgraph}
            Let $\mc{L} = \{R_i : i < \omega\}$, where each $R_i$ is a binary relation symbol. Let $A$ be an $\mc{L}$-structure where each $R_i^A$ is irreflexive and symmetric, and where for $i \neq j$ we have $R_i^A \cap R_j^A = \varnothing$. Then we call $A$ an \emph{$\omega$-coloured graph}. Note that it is permitted for two vertices to not have any coloured edge joining them.
        \end{defn}
			
        Let $\mc{K}$ be the class of finite $\omega$-coloured graphs. It is straightforward that $\mc{K}$ has free amalgamation, and we denote its \Fr limit by $M$, which we call the \emph{rainbow graph}. As for the random graph, the EP for $M$ is equivalent to the following \emph{witness property}: for each finite set of colours $I \fin \omega$, we have that
        \begin{enumerate}
            \item[($\ast_I$)] for all $U_i \fin M$, $i \in I$, and $V \fin M$, where these sets are pairwise disjoint, there exists $m \in M$ disjoint from $\bigcup_{i \in I} U_i \cup V$ such that:
				\begin{itemize}
				    \item for $i \in I$, $m$ has an edge of colour $i$ to each vertex of $U_i$;
					\item $m$ has no edges to any vertex of $V$.
				\end{itemize}
		\end{enumerate}
        Any two countable $\omega$-coloured graphs satisfying the witness property are isomorphic by a back-and-forth argument. But as $\Aut(M)$ has infinitely many orbits on $M^2$, by the Ryll-Nardzewski theorem $M$ is not $\aleph_0$-categorical.
    	
        \begin{lem}
		$\Th(M)$ has QE.
	\end{lem}
        \begin{proof}
            For $I \fin \omega$, let $\mc{L}_I = \{R_i : i \in I\}$, and let $M_I$ be the $\mc{L}_I$-reduct of $M$. Each $M_I$ satisfies $(\ast_I)$, so has the EP and is thus ultrahomogeneous in a finite relational language, and so $\Th(M_I)$ has QE. Thus $\Th(M)$ has QE.
        \end{proof}

        It is then quick to see that $M$ is $\aleph_0$-saturated: using QE, this follows from the definition of $\mc{K}$ and the EP of $M$.
        		
        $M$ has the pigeonhole property (see \Cref{pigeonindivdef}): the proof is similar to that for the random graph. Thus we see that $M$ has the vertex-Ramsey property by \Cref{pigeonindivlem,ageindivpointRamsey}, and as $\Th(M)$ has QE, by \Cref{EDRPsimplecheck} we have:
			
        \begin{prop}
            The $\aleph_0$-saturated rainbow graph $M$ has the externally definable Ramsey property.
        \end{prop}
		
    \subsection{The rainbow graph, unsaturated version}
		
        \begin{defn}
            Let $\mc{L} = \{R_i : i < \omega\}$ be as in the previous example. We say that an $\mc{L}$-structure $A$ is a \emph{complete $\omega$-coloured graph} if $A$ is an $\omega$-coloured graph (see \Cref{omegacolouredgraph}) and for all $a, b \in A$ with $a \neq b$, there exists $i < \omega$ with $a R_i b$.
        \end{defn}
        Let $\mc{K}$ be the class of all finite complete $\omega$-coloured graphs. It is easy to see that $\mc{K}$ is an amalgamation class: we denote its \Fr limit by $M$. The EP for $M$ is equivalent to the following witness property: for each finite set of colours $I \fin \omega$, we have that
        \begin{enumerate}
            \item[($\ast\ast_I$)] for all $U_i \fin M$ ($i \in I$), where these finite sets are pairwise disjoint, there exists $m \in M$ disjoint from $\bigcup_{i \in I} U_i$ such that for $i \in I$, $m$ has an edge of colour $i$ to each vertex of $U_i$.
        \end{enumerate}
		
        Any two countable complete $\omega$-coloured graphs satisfying this witness property are isomorphic via a back-and-forth argument. As in the previous example, $\Th(M)$ is not $\aleph_0$-categorical. To see that $M$ is not $\aleph_0$-saturated, observe that the set $\{\neg x R_i y : i < \omega\}$ is finitely satisfiable in $M$, but not satisfiable in $M$. We have that $\Th(M)$ has QE, and the proof is as for the saturated rainbow graph (note that here reducts will not be complete $\omega$-coloured graphs). As for the saturated rainbow graph, we see that $M$ has the pigeonhole property, and therefore by \Cref{EDRPsimplecheck} we have that:
			
        \begin{prop}
            The unsaturated rainbow graph $M$ has the externally definable Ramsey property.
        \end{prop} 
			
    \subsection{Infinitely refining equivalence relations, saturated version} \label{satinfrefinERs}
        Let $\mc{L} = \{R_i : i < \omega\}$ with each $R_i$ binary. Let $\mc{K}$ be the class of finite $\mc{L}$-structures $A$ where:
        \begin{enumerate}
            \item for $i < \omega$, $R_i^A$ is an equivalence relation;
            \item for $i < j < \omega$, $R_j^A \sub R_i^A$. 
        \end{enumerate}
			
        It is easy to see that $\mc{K}$ is an amalgamation class: given embeddings $A \to B, A \to C$ of $A, B, C \in \mc{K}$, take the free amalgam of $B, C$ over $A$ and then the transitive closure of each relation therein. Let $M$ denote the \Fr limit of $\mc{K}$. It is easy to see that the action $\Aut(M) \curvearrowright M^2$ has infinitely many orbits, and so by the Ryll-Nardzewski theorem $M$ is not $\aleph_0$-categorical.
	\begin{lem} \label{satinfrefinERsQE}
            $\Th(M)$ has QE.
	\end{lem}
	\begin{proof}
            The proof is as for the rainbow graph. For $i < \omega$, let $\mc{L}_i = \{R_j : j < i\}$, and let $M_i$ be the $\mc{L}_i$-reduct of $M$. We show that $M_i$ has the EP, is thus ultrahomogeneous in a finite relational language, and so $\Th(M_i)$ has QE. As $M$ has the EP, it suffices to show that for $A_i, B_i \in \Age(M_i)$ with $B_i = A_i \cup \{b\}$, for any $\mc{L}$-expansion $A \in \Age(M)$ of $A_i$, there exists an $\mc{L}$-expansion $B \in \Age(M)$ of $B_i$ with $A \sub B$. For $j < i$, let $R_j^B = R_j^{B_i}$, and for $j \geq i$, let the equivalence classes of $R_j^B$ consist of the equivalence classes of $R_j^A$ together with an additional singleton equivalence class $\{b\}$. Then $B$ is as required.
	\end{proof}

        To see that $M$ is $\aleph_0$-saturated, the proof is as for the rainbow graph.
        
        \begin{prop}
            The structure $M$ has the externally definable Ramsey property.
	\end{prop}
	\begin{proof}
            For $i < \omega$, let $\phi_i(x, y) \equiv x R_i y$. By \Cref{EDRPsimplecheck}, it suffices to check the vertex-Ramsey property for each $2$-colouring externally defined in an elementary extension by the formula $\phi_i(x, y)$ ($i < \omega$).
				
            Take $i < \omega$, and let $\chi_i : \binom{M}{a} \to 2$ be a vertex colouring externally defined by some $e \in N \succeq M$ and $\phi_i(x, y)$. Then the vertices of $M$ of colour $0$ (if indeed there are any) lie in a single $R_i$-equivalence class of $M$. Therefore, it suffices to show that for $m \in M, B \in \Age(M)$, there is $B' \in \binom{M}{B}$ with each vertex of $B'$ not in the $R_i$-equivalence class of $m$ in $M$. This follows immediately from the EP of $M$.
        \end{proof}

        In fact, by more general results, a stronger statement is true:

        \begin{prop}
            The structure $M$ is indivisible, and hence has the vertex-Ramsey property.
        \end{prop}

        To see this, we note that the saturated version of the structure $M$ with infinitely refining equivalence relations is isomorphic to the lexicographic product $M'[N]$, where $M'$ is the unsaturated version of the structure of infinitely refining equivalence relations described in the following section, and $N$ is a pure set in the language of equality. (See \Cref{lexsec}.) We will establish in \Cref{inf ref eq rel vertex-Ramsey} that $M'$ is indivisible, and $N$ is clearly indivisible as well. We then see from the following fact that $M'[N]$ is indivisible (since $M'[N]$ is a reduct of $M'[N]^s$), and thus in particular has the vertex-Ramsey property.

        \begin{fact}[{\cite[Proposition 2.21]{Mei16}}] \label{lex prod indivisible with s}
            If $M, N$ are indivisible structures, so is $M[N]^s$.
        \end{fact}
  
    \subsection{Infinitely refining equivalence relations, unsaturated version}\label{sec: inf ref eq rel unsat}
        Let $\mc{L} = \{R_i : i < \omega\}$ with each $R_i$ binary. Let $\mc{K}$ be the class of finite $\mc{L}$-structures $A$ where:
        \begin{enumerate}
            \item for $i < \omega$, $R_i^A$ is an equivalence relation;
            \item for $i < j < \omega$, $R_j^A \sub R_i^A$;
            \item there exists $l < \omega$ such that the equivalence classes of $R_l^A$ are all singletons. 
        \end{enumerate}
        
        As in the previous example, it is straightforward to see that $\mc{K}$ is an amalgamation class, and letting $M$ denote the \Fr limit of $\mc{K}$, as before $\Th(M)$ has QE. As for the unsaturated rainbow graph, $M$ is not $\aleph_0$-saturated.

        By an analogous argument to that for the previous saturated example, we have that:
        
        \begin{prop}\label{inf ref eq rel edrp}
            The unsaturated structure $M$ has the externally definable Ramsey property.
        \end{prop}

        In fact, we have a stronger statement.

        \begin{prop}\label{inf ref eq rel vertex-Ramsey}
            The structure $M$ is indivisible. In particular, it has the vertex-Ramsey property.
        \end{prop}

        This follows from \Cref{infProdIndOmega}: the proof of \Cref{inf ref eq rel vertex-Ramsey} relies on the fact that $M$ is a special case of the infinite lexicographic product, which we discuss in greater detail in \Cref{infinite products}. 
                
    \subsection{Disjoint random graphs} \label{disjrandgraphs}
        Let $\mc{L} = \{\sim\}$ be the language of graphs. Let $M = M_0 \cup M_1$ be the union of two disjoint copies $M_0, M_1$ of the random graph.

        We will see in this section that $M$ is an example of a structure which is not ultrahomogeneous, and without QE, which does have externally definable Ramsey. We will also see that $M$ has a $\varnothing$-definable expansion to an ultrahomogeneous structure $M_s$ which does not have externally definable Ramsey.

        To see that $M$ is not ultrahomogeneous, take $a, b \in M_0$ with $a \not\sim b$, and take $c \in M_1$. Then $f : \{a, b\} \to \{b, c\}$, $f(a) = c, f(b) = b$ is an isomorphism. But there is no automorphism of $M$ extending $f$, as $a, b$ have a common neighbour and $b, c$ do not.
			
        Let $\mc{L}_s = \mc{L} \cup \{s\}$ with $s$ binary. Let $M_s$ be the $\mc{L}_s$-structure given by the expansion of $M$ by the equivalence relation $s$ where vertices are equivalent iff they both lie in $M_0$ or both lie in $M_1$. It is easy to see that $M_s$ is ultrahomogeneous, so by \Cref{finLomega,finLQE}, we have that $\Th(M_s)$ is $\aleph_0$-categorical and has QE. So by the Ryll-Nardzewski theorem, $\Aut(M_s)$ is oligomorphic, and so $\Aut(M)$ is oligomorphic and thus $\Th(M)$ is $\aleph_0$-categorical. By \Cref{omegahomogiffQE}, as $M$ is not ultrahomogeneous, $\Th(M)$ does not have QE.
			
        \begin{prop}
            The $\mc{L}_s$-structure $M_s$ does not have the externally definable Ramsey property.
        \end{prop}
        \begin{proof}
            Take $e \in M_0$, and let $\phi(x, y) \equiv x\,\text{s}\,y$. Let $\chi: \binom{M_s}{a} \to 2$ be the vertex-colouring of $M_s$ externally defined by $e \in M_s$ and $\phi(x, y)$ (where the elementary extension of $M_s$ that we take is $M_s$ itself). Then $M_i = \chi^{-1}(i)$, $i = 0, 1$. 
				
            Take $b_0 \in M_0, b_1 \in M_1$, and let $B = \{b_0, b_1\}$ be the substructure of $M_s$ induced on $\{b_0, b_1\}$, so $(b_0, b_1) \not\in s$. Then there is no $\chi$-monochromatic copy of $B$ in $M_s$. 
        \end{proof}

        \begin{prop}
            The $\mc{L}$-structure $M$, the disjoint union of two random graphs, does not have FPT.
        \end{prop}
        \begin{proof}
            It is straightforward to see that if a structure has FPT, then so does any $\varnothing$-definable relational expansion. As $M_s$ is a $\varnothing$-definable relational expansion of $M$ without FPT, we have that $M$ does not have FPT.
        \end{proof}

        \begin{prop}
            The $\mc{L}$-structure $M$, the disjoint union of two random graphs, has the externally definable Ramsey property.
        \end{prop}
        \begin{proof}
            Take the $\mc{L}$-formula $\sigma(x, y) \equiv (\ex z)(z \sim x \wedge z \sim y)$. Then the equivalence relation $s$ is $\varnothing$-definable in $M$ by $\sigma(x, y)$. As $\Th(M_s)$ has QE, we have that $\Th(M)$ has ``QE up to $\sigma(x, y)$", i.e.\ each $\mc{L}$-formula is equivalent modulo $\Th(M)$ to a Boolean combination of $x \sim y$ and $\sigma(x, y)$. Therefore, by \Cref{boolcomb} and a straightforward adaptation of \Cref{EDRPt}, it suffices to check the externally definable RP for $M$ for colourings externally defined by $x \sim y$ and by $\sigma(x, y)$. For a colouring $\chi : \binom{M}{a} \to 2$ externally defined by $e \in N \succeq M$ and $\phi(x, y)$, where $\phi(x, y) \equiv x \sim y$ or $\phi(x, y) = \sigma(x, y)$, we have that $\chi$ restricts to a colouring on $M_0$, and as $\Age(M) \sub \Age(M_0)$ and $M_0$ has the point-Ramsey property, we are done.
        \end{proof}

\section{The lexicographic product} \label{lexsec}
    
    \begin{defn} \label{lexdefn}
        Let $M, N$ be $\mc{L}$-structures. The \emph{lexicographic product} $M[N]$ is the $\mc{L}$-structure with domain $M \times N$ where for each $n$-ary relation symbol $R \in \mc{L}$, we define that the $n$-tuple $((a_1, b_1), \cdots, (a_n, b_n))$ lies in $R^{M[N]}$ if one of the following holds:
            \begin{itemize}
                \item there exist $i \neq j$ with $a_i \neq a_j$, and $M \models (a_1, \cdots, a_n)$; or
                \item $a_1 = \cdots = a_n$ and $N \models (b_1, \cdots, b_n)$.
            \end{itemize}

        If $M$ and $N$ are structures in different languages, we take the lexicographic product by taking the union of the respective languages, then expand each structure by taking the new relations to be $\varnothing$, and subsequently take the lexicographic product.
    \end{defn}

    This construction generalises the lexicographic order and the lexicographic product of graphs. More generally, in a binary language, $M[N]$ coincides with a classical construction denoted by the same notation.
  
    Let $\mc{L}^s$ be the expansion of $\mc{L}$ by a new binary relation symbol $s$, and let $M[N]^s$ denote the $\mc{L}^s$-structure which is the expansion of the $\mc{L}$-structure $M[N]$ by:
    \[M[N]^s \models s((a_1, b_1), (a_2, b_2)) \iff a_1 = a_2.\] Let $N^s$ denote the $\mc{L}^s$-structure resulting from expanding $N$ by the binary relation $s = N^2$. Then $M[N]^s = M[N^s]$.

    We recall the following from \cite{Mei16}.
	
    \begin{prop}[{\cite[Proposition 2.23]{Mei16}}]\label{boolean combination in product}
        Let $M, N$ be $\mc{L}$-structures, and assume $\Th(M)$ is transitive (i.e.\ any two points of $M$ are isomorphic as substructures).
        
        Let $\phi(\bar{x})$ be an $\mc{L}^s$-formula. Then there exist $k \in \N$ and $\mc{L}$-formulae $(\phi_1^j(\bar{x}))_{j < k}, (\phi_2^j(\bar{x}))_{j < k}$ such that, for every tuple $\ov{(a, b)} \in M[N]^s$ with $|\ov{(a, b)}| = |\bar{x}|$, we have that:
        \[M[N]^s \models \phi(\ov{(a, b)}) \iff \bigvee_{j < k} (M \models \phi_1^j(\ov{a}) \wedge N \models \phi_2^j(\ov{b})).\]
    \end{prop}

    \begin{lem}\label{types in prod depend on types in factors}
        Let $M, N$ be $\mc{L}$-structures, and assume $\Th(M)$ is transitive. Let $A\subseteq M$ and $B\subseteq N$ with $A, B \neq \varnothing$. Let $\bar{c},\bar{c}'\in M$ and $\bar{d},\bar{d}'\in N$. Suppose that $\tp(\ov{c}/A) = \tp(\ov{c}'/A)$ and $\tp(\ov{d}/B) = \tp(\ov{d}'/B)$. Then $\tp(\ov{(c,d)}/A[B]^s) = \tp(\ov{(c',d')}/A[B]^s)$.
    \end{lem}
    \begin{proof}
        Let $\phi(\ov{x},\ov{z})$ be an $\mc{L}^s$-formula such that $|\ov{z}| = |\ov{(c,d)}|$. By \Cref{boolean combination in product}, there exist $k \in \N$ and $\mc{L}$-formulae $(\phi_1^j(\ov{x},\ov{z}))_{j < k}, (\phi_2^j(\ov{x},\ov{z}))_{j < k}$ such that, for every pair of tuples $\ov{(\alpha, \beta)},\ov{(\gamma,\delta)} \in M[N]^s$ with $|\ov{(\alpha, \beta)}| = |\ov{x}|$ and $|\ov{(\gamma, \delta)}| = |\ov{z}|$, we have: 
        \[M[N]^s \models \phi(\ov{(\alpha, \beta)},\ov{(\gamma,\delta)}) \iff \bigvee_{j < k} (M \models \phi_1^j(\ov{\alpha},\ov{\gamma}) \wedge N \models \phi_2^j(\ov{\beta},\ov{\delta})).\]
        Now let $\ov{(a,b)}\in A[B]^s$ with $|\ov{(a,b)}|=|\ov{x}|$. Then
        \begin{align*}
            M[N]^s\models \phi(\ov{(a,b)},\ov{(c,d)})   \iff 
               &  \bigvee_{j < k} (M \models \phi_1^j(\ov{a},\ov{c}) \wedge N \models \phi_2^j(\ov{b},\ov{d}))  \\
            \iff & \bigvee_{j < k} (M \models \phi_1^j(\ov{a},\ov{c'}) \wedge N \models \phi_2^j(\ov{b},\ov{d'})) \\
            \iff&  M[N]^s\models \phi(\ov{(a,b)},\ov{(c',d')}).
        \end{align*}
    \end{proof}    
    
    \begin{lem}[{\cite[Lemma 2.13]{Mei16}}]\label{automorphism of wreath product}
        Let $M, N$ be $\mc{L}$-structures. Let $f \in \Aut(M)$, and for each $a \in M$, let $g_a \in \Aut(N)$. Then the function $F : M[N]^s \to M[N]^s$ given by $F(x, y) = (f(x), g_x(y))$) is an automorphism of $M[N]^s$.

        Conversely, if $\Th(M)$ is transitive, then for every $F \in \Aut(M[N]^s)$, there exists an automorphism $f \in \Aut(M)$ and a family of automorphisms $(g_a \in \Aut(N) : a \in M)$ such that for all $(x, y) \in M[N]^s$, $F(x, y) = (f(x), g_x(y))$. 
        
        In other words, if $\Th(M)$ is transitive, then \[\Aut(M[N]^s) = \Aut(M) \Wr_M \Aut(N),\] where $\Wr$ denotes the wreath product.
    \end{lem}

    \begin{lem}[{\cite[Lemma 1.6.5]{Mei19}}]\label{saturation in product}
        Let $M,N$ be $\mc{L}$-structures, and let $\kappa$ be a cardinal. The following are equivalent:
        \begin{enumerate}
            \item $M, N$ are $\kappa$-saturated.
            \item $M[N]^s$ is $\kappa$-saturated.
        \end{enumerate}
    \end{lem}

   \begin{lem}[{\cite[Corollary 1.2.5]{Mei19}}]\label{elementary equivalence product}
        Let $M, M', N, N'$  be $\mc{L}$-structures. If $M\equiv M'$ and $N\equiv N'$, then $M[N]^s \equiv M'[N']^s$.
    \end{lem}
    
    \begin{lem}\label{elementary extension product}
        Let $M, N$ be $\mc{L}$-structures, and let $E \succeq M[N]^s$ be an elementary extension. Then there are elementary extensions $M' \succeq M$ and $N' \succeq N$ such that $M'[N']^s \succeq E$.
        
        \begin{proof}
            Take $\kappa:=|E|$ and let $M'\succeq M$ and $N'\succeq N$ such that $M',N'$ are $\kappa^+$-saturated. Then $M'[N']^s$ is $\kappa^+$-saturated. By \Cref{elementary equivalence product}, $M'[N']^s\equiv M[N]^s\equiv E$, so, by saturation, $M'[N']^s\succeq E$.
        \end{proof}
    \end{lem}
    
    \begin{prop}\label{fpt in fin prod}
        Let $\Th(M)$ be transitive. If $M$ has FPT$_n$, then so does $M[N]^s$.
    \end{prop}

    \begin{proof}
        Let $G = \Aut(M[N]^s)$.
    
        Let $p \in S_n(M[N]^s)$. We realise $p$ as a tuple $\ov{(a, b)}$ in an elementary extension of $M[N]^s$, which, by \Cref{elementary extension product}, we may assume is $M'[N']^s$ for some $M' \succeq M$, $N' \succeq N$. We will also assume that $N'$ is strongly $(|N| + \aleph_0)^+$-homogeneous.

        Let $p_1 = \tp_{M'}(\ov{a}/M)$ and $p_2 = \tp_{N'}(\ov{b}/N)$. As $M$ has $FPT_n$, there is an $\Aut(M)$-invariant type $q_1 \in \ov{\Aut(M) \cdot p_1}$. Let $\bar{c} \in M'' \succeq M$ be a realisation of $q_1$, where $M''$ is some elementary extension of $M$. Since $\Th(M)$ is transitive and $q_1$ is $\Aut(M)$-invariant, we have that $\bar{c}\in M''\setminus M$. Let $q = \tp_{M''[N']^s}(\ov{(c, b)}/M[N]^s)$. By \Cref{types in prod depend on types in factors}, $q$ depends only on $q_1$ and $p_2$, so we denote $q$ obtained as above from $q_1$ and $p_2$ by $q=:q_1[p_2]$. We claim that $q$ is $G$-invariant and lies in the $G$-orbit closure of $p$.

        We first show that $q$ is $G$-invariant. Take $\sigma \in G$. By \Cref{automorphism of wreath product}, there exist $f \in \Aut(M)$ and a family of automorphisms $(g_m \in \Aut(N) : m \in M)$ such that $\sigma(x, y) = (f(x), g_x(y))$. Let $\sigma_1(x, y) = (f(x), y)$ and let $\sigma_2(x, y) = (x, g_x(y))$, so $\sigma = \sigma_1 \circ \sigma_2$, and therefore it suffices to show that $q$ is $\sigma_i$-invariant for $i \in \{1, 2\}$.

        Using the fact that $N'$ is strongly $(|N| + \aleph_0)^+$-homogeneous, extend each $g_m$ (for $m \in M$) to an automorphism $g'_m$ of $N'$. By \Cref{automorphism of wreath product}, the function $\sigma''_2 : M''[N']^s \to M''[N']^s$ given by
        \[ \sigma''_2(x, y) =
        \begin{cases}
            (x, g'_x(y)) & x \in M \\
            (x, y) & x \not\in M
        \end{cases}
        \]
	   is an automorphism of $M''[N']^s$. We have that $\sigma''_2$ extends $\sigma_2$, and using the fact that $\bar{c}\in M''\setminus M$, we have $\sigma''_2(\ov{(c, b)}) = \ov{(c, b)}$. We therefore have that $q$ is $\sigma_2$-invariant.

        To see that $q$ is $\sigma_1$-invariant, we recall that $q_1$ is $\Aut(M)$-invariant, so in particular is invariant under $f$. Therefore, $\sigma_1(q) = \sigma_1(q_1[p_2]) =  f(q_1)[p_2] = q_1[p_2]= q$. This completes the proof of the claim that $q$ is $G$-invariant.

        To show that $q\in \overline{G\cdot p}$, let $\phi$ be an $\mc{L}^s$-formula such that $\phi \in q$. We must find some $g\in G$ such that $\phi \in g\cdot p$. 
        
        Since $q=q_1[p_2]$, by definition and by \Cref{types in prod depend on types in factors}, there are $\mc{L}$-formulae $\phi_1,\phi_2$ with $\phi_1\in q_1$ and $\phi_2\in p_2$ such that whenever $q' = q'_1[q'_2]$ with $\phi_1\in q'_1$ and $\phi_2\in q'_2$, we have that $\phi\in q$.
        Now, $q_1\in \overline{\Aut(M)\cdot p_1}$, and therefore there is some $g_1\in \Aut(M)$ such that $\phi_1\in g_1\cdot p_1$. Finally, let $g\in \Aut(M[N]^s)$ be defined by
        $g(x,y):=(g_1(x),y)$. 
        By \Cref{automorphism of wreath product}, 
        we have $g\in \Aut(M[N]^s)$,
        and 
        $g \cdot p = g\cdot p_1[p_2] = (g_1\cdot p_1)[p_2]\ni \phi$.
    \end{proof}

\subsection{The infinite lexicographic product} \label{infinite products}

    In this \namecref{infinite products}, we review the infinite lexicographic product from \cite{Me2} and demonstrate the use of our main theorem (\Cref{FPTiffEDRP}) in this particular case.
	
    The following definition is a special case of {\cite[Example 4.6]{Me2}}.
    \begin{defn}\label{OmegaLOmega product}
        Let $\left(M_i\right)_{i\in \omega}$ be a family of countably infinite $\mc{L}$-structures, each with domain $\omega$. The infinite lexicographic product $[M_i]_{i < \omega}$ is the $\mc{L}$-structure with domain $\omegaomega$ where, for every $a_1,\dots, a_k\in \omegaomega$ and every $k$-ary relation $R\in \mc{L}$, letting $n\in\omega$ be maximal such that $a_1 \upharpoonright n = \dots = a_k \upharpoonright n$, we have that 
        \[[M_i]_{i < \omega} \models R(a_1,\dots, a_k) \iff M_n \models R(a_1(n),\dots, R(a_k(n)).\]
    \end{defn}

    While most of the results in the following can be stated more generally (see \cite[Definition 4.5]{Me2}), for our purposes, we restrict our attention to products as in \Cref{OmegaLOmega product}.

    Throughout this \namecref{infinite products}, we fix a family $(M_i)_{i\in \omega}$ of \Fr structures, where each $M_i$ is transitive and has domain $\omega$. (This is called $\omega$-pure in \cite[Definition 5.2]{Me2}.)

    We fix a family $\{s_n \mid n<\omega\}$ of binary relation symbols, and let $M_i^s$ be the expansion of $M_i$ by $\{s_n \mid n<\omega\}$ such that
    $\left(s_n\right)^{M_i^s}=\twopartdef{\left(M_i\right)^2}{i> n}{\varnothing}{i\leq n.}$.

    \begin{rem}[{\cite[Remark 5.4]{Me2}}]\label{remark: expansion of inf product}
        With the expansion above, $[M_i^s]^{}_{i < \omega}$ has refining equivalence relations $s_0\supseteq s_1 \supseteq s_2 \supseteq \cdots$. We also have that $[M_i^s]^{}_{i < \omega}$ is the expansion of $[M_i]_{i < \omega}$ by  $\{s_n | n<\omega\}$  such that
	$s_n(x,y)\Leftrightarrow x(i) = y(i)$ for all $i \leq n$. 
    \end{rem}

    The following \namecref{inf product is fin} is a special case of {\cite[Lemma 4.15]{Me2}}, and its proof is straightforward.
    
    \begin{lem}\label{inf product is fin}
        For $j \geq 1$, let $N_j = M_j$. Let $N^s = [N_j^s]^{}_{j \geq 1}$ (here the enumeration of the $s_n$ begins at $n = 1$). Then we have $M_0[N^s]^{s_0} \cong [M_i^s]^{}_{i < \omega}$ via the map $(a,b) \mapsto a\smallfrown~b$ (the concatenation of $a$ and $b$).
    \end{lem}
    
    Now that we have established the basic setting of the infinite products relevant to our discussion, we recall a few facts from \cite{Me2} regarding infinite lexicographic products.

    \begin{defn}
        A substructure $D\subseteq [M^s_i]^{}_{i < \omega}$ is \emph{dense} if for all $t\in \omegaLomega$, there is some $d\in D$ such that $d\supseteq t$.
    \end{defn}

    \begin{fact}[{\cite[Theorem 5.14]{Me2}}]\label{Dense substructure unique and ultrahomogeneous}\ 
        \begin{enumerate}
		\item Up to isomorphism, there is a unique countable dense substructure $D\subseteq [M_i^s]^{}_{i < \omega}$.
		\item\label{denseUniqueHom} Such a $D$ is transitive and ultrahomogeneous.
        \end{enumerate}
    \end{fact}

    \begin{fact}[{\cite[Theorem 6.1]{Me2}}]\label{infProdIndOmega}
        If  $D\subseteq [M_i^s]^{}_{i < \omega}$ is a countable dense substructure, then $D$ is indivisible.
    \end{fact}

    We will use \Cref{Dense substructure unique and ultrahomogeneous} to prove the following \namecref{inf prod edrp iff M edrp}.

    \begin{thm}\label{inf prod edrp iff M edrp}
        Let $D\subseteq [M_i^s]^{}_{i < \omega}$ be a countable dense substructure. Then the following are equivalent:
        \begin{enumerate}
            \item\label{item:D edrp} $D$ has the externally definable Ramsey property.
            \item\label{item:D fpt} $D$ has FPT.
            \item\label{item:M0 edrp} $M_0$ has the externally definable Ramsey property.
            \item\label{item:M0 fpt} $M_0$ has FPT.
        \end{enumerate} 
    \end{thm}

    In the below \Cref{dense substructure is fin product of dense} and \Cref{inf prod fpt iff M fpt}, we use the notation of \Cref{inf product is fin}.
    
    \begin{lem}\label{dense substructure is fin product of dense}
        Let $D\sub [M_i^s]_{i < \omega}$ and $D' \sub N^s$ be countable dense substructures. Then $D\cong M_0[D']^{s_0}$.

        \begin{proof}
            As $M_0[D']^{s_0}$ is countable and dense in $M_0[N^s]^{s_0}$, by \Cref{inf product is fin} and \Cref{Dense substructure unique and ultrahomogeneous} we have $D\cong M_0[D']^{s_0}$.
        \end{proof}
    \end{lem}

    \begin{prop}\label{inf prod fpt iff M fpt}
        Let $D\sub [M_i^s]^{}_{i < \omega}$ be a countable dense substructure. Suppose that $M_0$ has FPT$_n$. Then $D$ has FPT$_n$.

        \begin{proof}
            Let $D'\sub N^s$ be countable and dense. By \Cref{dense substructure is fin product of dense}, $D\cong M_0[D']^{s_0}$. By \Cref{fpt in fin prod}, as $M_0$ has FPT$_n$, we have that $M_0[D']^{s_0}$ has FPT$_n$.
        \end{proof}
    \end{prop}

    \begin{proof}[Proof of \Cref{inf prod edrp iff M edrp}]
        By \Cref{inf prod fpt iff M fpt}, $D$ has FPT if and only if $M_0$ has FPT.
        By \Cref{Dense substructure unique and ultrahomogeneous}, $D$ is ultrahomogeneous. Recall that $M_0$ was also assumed to be ultrahomogeneous and countable. Therefore, by \Cref{FPTiffEDRP}, we have that $D$ has the externally definable Ramsey property if and only if $D$ has FPT, and $M_0$ has the externally definable Ramsey property if and only if $M_0$ has FPT.
    \end{proof}

    \begin{eg}
        If $M_i$ is an infinite set in the language of pure equality, for all $i\in \omega$, then a countable dense substructure $D\sub [M_i^s]^{}_{i < \omega}$ is the structure with infinitely many refining equivalence relations (unsaturated version) defined in \Cref{sec: inf ref eq rel unsat}.
    \end{eg}

    We end this \namecref{infinite products} by showing that a dense substructure of an infinite lexicographic product is never $\aleph_0$-saturated:

    \begin{prop}\label{inf prod not A-sat}
        $D$ is not $\aleph_0$-saturated.
    \end{prop}
    \begin{proof}
        The partial type $\{x\neq y\}\,\cup\,\{s_i(x,y) \mid i\in \omega\}$ is finitely satisfiable but not satisfiable in $D$.
    \end{proof}

    \begin{prop}\label{continuoum many edrp}
        Let $\mc{L}$ be a countable relational language, and let $T$ be a universal $\mc{L}$-theory closed under lexicographic products, i.e.\ if $M,N\models T$ then $M[N]\models T$. Let $T^s$ be the $\mc{L} \cup \{s_i \mid i\in \omega\}$-theory defined as $T$ together with the $\{s_i \mid i\in \omega\}$-theory of infinitely many refining equivalence relations with infinite classes.
    
        Suppose that $T$ has at least two countable transitive ultrahomogeneous models, one of which has the externally definable Ramsey property. Then the theory $T^s$ has $2^{\aleph_0}$ many pairwise non-isomorphic, countable, transitive ultrahomogeneous models with the externally definable Ramsey property which are not $\aleph_0$-saturated.
    \end{prop}
    \begin{proof}
        Take $A_0,A_1$ to be two distinct countable transitive ultrahomogeneous models of $T$, such that $A_0$ has the externally definable RP. For each $a\in 2^{\omega\setminus \{0\}}$, let $M(a)_0 = A_0$ and $M(a)_j = A_{a(j)}$ for $j \geq 1$. Let $M(a)$ be a countable dense substructure of $[M(a)_i^s]^{}_{i < \omega}$. By \Cref{Dense substructure unique and ultrahomogeneous}, $M(a)$ is ultrahomogeneous. By \Cref{inf prod edrp iff M edrp}, $M(a)$ has the externally definable Ramsey property. By \Cref{inf prod not A-sat}, $M(a)$ is not $\aleph_0$-saturated.
        
        It remains to show that $M(a)\not\cong M(b)$ for distinct $a, b \in 2^{\omega\setminus \{0\}}$.  Assume for a contradiction that $f:M(a)\to M(b)$ is an isomorphism. Let $m < \omega$ be least such that $a(m+1)\neq b(m+1)$. By density of $M(a)$, for each $i < \omega$ there is $\sigma_i \in M(a)$ with initial segment $(a(0), \cdots, a(m), i)$. Notice that the $\sigma_i$ are pairwise $s_m$-equivalent, but not $s_{m+1}$-equivalent, and, in addition, they $s_{m+1}$-cover their equivalence class. As $f$ is an $\mc{L}\cup\{s_i \mid i < \omega\}$-isomorphism, the same properties are also satisfied by the $f(\sigma_i)$ in $M(b)$. Finally, notice that, by the definition of the infinite lexicographic product, the substructure induced on $\{\sigma_i\}_{i\in \omega}$ in $M(a)$ is isomorphic to $M(a)_{m+1}$ and the substructure induced on $\{f(\sigma_i)\}_{i\in \omega}$ in $M(b)$ is isomorphic to $M(b)_{m+1}$. But $M(b)_{m+1}\cong A_{b(m+1)}\not\cong A_{a(m+1)}\cong M(a)_{m+1}$, contradiction.
    \end{proof}

    \begin{eg}
        Examples of $T$ satisfying the assumptions of \Cref{continuoum many edrp} include:
        \begin{itemize}
            \item the theory of $n$-hypergraphs, for $n\geq 2$;
            \item the theory of $K_m$-free $n$-hypergraphs for $m\geq n\geq 3$;
            \item the theory of partially ordered sets.
        \end{itemize}
    \end{eg}

\section{Further questions} \label{furtherqnssec}

    Firstly, we have two specific questions regarding the basic lemmas of \Cref{basicpropertiessec} and \Cref{FPTsec}:

    \begin{qn}
        Does there exist a \Fr structure $M$ having $\bar{a}, B \in \Age(M)$ such that $M \to_{\ext} (B)^{\bar{a}}_2$, but where there does not exist $C \in \Age(M)$ with $C \to_{\ext} (B)^{\bar{a}}_2$?
    \end{qn}
    \begin{qn}[see \Cref{FPTn+1impliesn}]
        Is there a \Fr structure $M$ which has FPT$_n$ for some $n \in \N_+$, but which does not have FPT$_{n+1}$?
    \end{qn}

    A positive answer to the first question (which we regard as probable) would show that the condition of finiteness of the set of formulae in \Cref{Ramseyforage} is indeed necessary.
    
    We also have two questions regarding whether the assumptions in \Cref{FPTiffEDRP} are necessary:

    \begin{qn}
        Does there exist an ultrahomogeneous structure $M$ with uncountable age such that $M$ has the externally definable Ramsey property but does not have FPT?
    \end{qn}
    \begin{qn}
        Does there exist an example of a non-ultrahomogeneous structure $M$ which has FPT but does not have the externally definable Ramsey property?
    \end{qn}

    Considering the Ramsey classification programme of \Jarik (\cite{Nes05}) and the central results of \cite{KPT05}, \cite{NVT13}, we have a number of general questions regarding externally-definable analogues.

    \begin{qn}
        Is it possible to classify \Fr structures with the externally definable Ramsey property, given certain restrictive assumptions (e.g.\ in a binary language or a finite relational language)? 
    \end{qn}
    \begin{qn}
        Does every \Fr structure in a finite relational language have a finite relational expansion with the externally definable Ramsey property?
    \end{qn}
    \begin{qn}
        If a \Fr structure $M$ has a precompact expansion with the externally definable Ramsey property, what does this imply about the topological dynamics of $\Aut(M)$?
    \end{qn}

    We would also be interested in seeing a wider range of examples of structures with the externally definable Ramsey property, in order to better understand its combinatorial and model-theoretic character. Specifically:

    \begin{qn}
        In this paper, for most examples given of \Fr structures $M$ with the externally definable Ramsey property, we show this property by checking the $\bar{a}$-RP for all colourings, for $|\bar{a}| < t$ (where $t$ is the maximum arity of $\mc{L}$ and we have QE). Can we find interesting examples of \Fr structures $M$ where $M$ has the externally definable Ramsey property, but does not have the $\bar{a}$-RP for $|\bar{a}| < t$? (See \Cref{singleeqrel}.)
    \end{qn}
    \begin{qn}
        Are there natural examples of \Fr structures without QE which have the externally definable Ramsey property?
    \end{qn}
    \begin{qn}
        Are there other natural examples of structures which are not ultrahomogeneous and which have the externally definable Ramsey property? (See \Cref{disjrandgraphs} and \Cref{RPimpliesAP}.)
    \end{qn}

    We would also like to know the following regarding the fixed points on type spaces property (FPT):
    \begin{qn}
        Do there exist \Fr structures $M, M'$ (possibly in different languages), with $\Aut(M), \Aut(M')$ isomorphic as topological groups, such that $M$ has FPT and $M'$ does not?
    \end{qn}

\bibliographystyle{abbrv}
\bibliography{super}

\end{document}